\documentclass[
11pt,                          
final,                         
english                        
]{article}

\usepackage{amsmath}           
\usepackage{amssymb}           
\usepackage[utf8]{inputenc}    
\usepackage[T1]{fontenc}       
\usepackage[a4paper]{geometry} 
\usepackage{setspace}          
\usepackage{bbm}               
\usepackage{stmaryrd}          
\usepackage{mathtools}         
\usepackage[amsmath,thmmarks,hyperref]{ntheorem} 
\usepackage{exscale}           
\usepackage[sort]{cite}        
\usepackage{xspace}            
\usepackage[activate={true,nocompatibility},
            final,tracking=true,kerning=true,
            spacing=true,factor=1100,
            stretch=10,shrink=10,expansion=false
           ]{microtype}
\microtypecontext{spacing=nonfrench}  
\usepackage[hyperref]{xcolor}
\definecolor{mydarkblue}{RGB}{0,0,155}
\usepackage[final=true,         
           colorlinks=true,    
           allcolors=mydarkblue,  
           hypertexnames=false,
           plainpages=false,              
           pdfpagelabels=true,            
           pdfencoding=auto,              
           unicode=true                   
                      ]{hyperref}         
\usepackage[shortcuts]{extdash}  
\usepackage{tikz-cd}

\author{
  \textbf{Matthias Schötz}\thanks{Instytut Matematyczny PAN; ul. Śniadeckich 8, 00-656 Warsaw, Poland; \href{mailto:schotz@impan.pl}{schotz@impan.pl}}
}

\theoremheaderfont{\normalfont\bfseries}
\theorembodyfont{\itshape}
\newtheorem{lemma}{Lemma}

\newtheorem{proposition}[lemma]{Proposition}
\newtheorem{theorem}[lemma]{Theorem}
\newtheorem{corollary}[lemma]{Corollary}
\newtheorem{definition}[lemma]{Definition}

\theorembodyfont{\rmfamily}

\newtheorem{example}[lemma]{Example}
\newtheorem{remark}[lemma]{Remark}

\theoremheaderfont{\normalfont\bfseries}
\theorembodyfont{\normalfont}
\theoremstyle{nonumberplain}
\theoremseparator{:}
\theoremsymbol{\hbox{$\boxempty$}}
\newtheorem{proof}{Proof}

\newcommand{\RR}{\mathbbm{R}}
\newcommand{\NN}{\mathbbm{N}}
\newcommand{\Unit}{\mathbbm{1}}
\newcommand{\ZZ}{\mathbbm{Z}}
\newcommand{\FF}{\mathbbm{F}}
\newcommand{\set}[3][]{#1\{\,{#2}\;#1|\;{#3}\,#1\}}
\newcommand{\abs}[2][]{#1|{#2}#1|}

\newcommand{\monoid}[1]{\mathrm{#1}}
\newcommand{\group}[1]{\mathrm{#1}}
\newcommand{\ex}{\mathrm{ex}}

\newcommand{\canleq}{\lesssim}
\newcommand{\canequiv}{\approx}
\newcommand{\genleq}{\preccurlyeq}

\newcommand{\at}[2][]{#1|_{#2}}
\newcommand{\genGrp}[2][]{#1\langle\!#1\langle\,#2\,#1\rangle\!#1\rangle_{\mathrm{grp}}}
\newcommand{\genMonoid}[2][]{#1\langle\!#1\langle\,#2\,#1\rangle\!#1\rangle_{\mathrm{mnd}}}
\newcommand{\op}{{\mathrm{op}}}
\newcommand{\argument}{\,\cdot\,}

\DeclareMathOperator{\Grot}{Gr}
\newcommand{\unperf}{\mathrm{up}}
\newcommand{\GrotRedFirst}{\nabla_{\!1}}
\newcommand{\GrotRedSnd}{\nabla_{\!2}}
\newcommand{\subGrot}{{\mathrm{Gr}}}
\newcommand{\subRedFirst}{{\nabla\!,1}}
\newcommand{\subRedSnd}{{\nabla\!,2}}
\newcommand{\iotaGrot}{\iota_\subGrot}
\newcommand{\iotaRedFirst}{\iota_\subRedFirst}
\newcommand{\iotaRedSnd}{\iota_\subRedSnd}

\DeclareFontFamily{U}{mathx}{}
\DeclareFontShape{U}{mathx}{m}{n}{<-> mathx10}{}
\DeclareSymbolFont{mathx}{U}{mathx}{m}{n}
\DeclareMathAccent{\widehat}{0}{mathx}{"70}
\DeclareMathAccent{\widecheck}{0}{mathx}{"71}

\title{Associativity and Commutativity of Partially Ordered Rings}

\date{June 2024}

\begin{document}
\begin{onehalfspace}

\maketitle
\begin{abstract}
  Consider a commutative monoid $(\monoid M,+,0)$ and a biadditive binary operation $\mu \colon \monoid{M} \times \monoid{M} \to \monoid{M}$.
  We will show that under some additional general assumptions, the operation $\mu$ is automatically both associative and commutative.
  The main additional assumption is localizability of $\mu$, which essentially means that a certain canonical order on $\monoid M$ is compatible
  with adjoining some multiplicative inverses of elements of $\monoid{M}$.
  As an application we show that a division ring $\FF$ is commutative
  provided that for all $a \in \FF$ there exists a natural number $k$ such that $a-k$ is not a sum of products of squares.
  This generalizes the classical theorem that every archimedean ordered division ring is commutative to a more general class of 
  formally real division rings that do not necessarily allow for an archimedean (total) order.
  Similar results about automatic associativity and commutativity are well-known for special types of partially ordered extended rings
  (``extended'' in the sense that neither associativity nor commutativity of the multiplication is required by definition),
  namely in the uniformly bounded and the lattice-ordered cases, i.e.~for (extended) $f$\=/rings. In these cases the commutative
  monoid in question is the positive cone of the partially ordered extended ring. We also discuss how these classical results can be
  obtained from our main theorem.
\end{abstract} \ \\[-0.5cm]
{\small
	\textbf{2020 Mathematics Subject Classification:} 06F25, 12E15, 12J15\\
	\textbf{Keywords:} partially ordered ring, $f$-ring, division ring, formally real, associativity, commutativity
}

\section{Introduction}
In some simple examples it is quite easy to check that (seemingly) rather weak additional assumptions are sufficient in order to guarantee that
a biadditive binary operation $\mu \colon \monoid{M} \times \monoid{M} \to \monoid{M}$ on a commutative monoid $(\monoid{M}, +, 0)$ is
both associative and commutative:

Write $\NN_0 = \{0,1,2,\dots\}$ for the commutative monoid of natural numbers with the usual addition and take the
cartesian product $(\NN_0)^n$ for $n\in \NN_0$ and with elementwise addition. For $i \in \{1,\dots,n\}$ write 
$\delta_i \in (\NN_0)^n$ for the tuple of natural numbers that is given by $(\delta_i)_i \coloneqq 1$ and
$(\delta_i)_j \coloneqq 0$ otherwise, and set $\Unit \coloneqq \sum_{i=1}^n \delta_i$.
Let $\mu \colon (\NN_0)^n \times (\NN_0)^n \to (\NN_0)^n$ be a biadditive binary operation and assume that $\mu(\Unit,a) = a = \mu(a,\Unit)$ for all $a\in (\NN_0)^n$.
Given $i,j\in \{1,\dots,n\}$, then $\mu(\delta_i,\delta_j) + \sum_{k\in \{1,\dots,n\}\setminus \{i\}} \mu(\delta_k,\delta_j) = \mu(\Unit,\delta_j) = \delta_j$ implies $\mu(\delta_i,\delta_j)_\ell = 0$
for all $\ell \in \{1,\dots,n\} \setminus \{j\}$, and
$\mu(\delta_i,\delta_j) + \sum_{k\in \{1,\dots,n\}\setminus \{j\}} \mu(\delta_i,\delta_k) = \mu(\delta_i,\Unit) = \delta_i$ implies $\mu(\delta_i,\delta_j)_\ell = 0$
for all $\ell \in \{1,\dots,n\} \setminus \{i\}$. So if $i \neq j$, then $\mu(\delta_i,\delta_j) = 0$, and consequently $\mu(\delta_i,\delta_i) = \delta_i$
in the remaining case $i=j$. It follows that $\mu$ is just the elementwise multiplication of natural numbers, which is
of course an associative and commutative operation. If we drop the assumption $\mu(\Unit,a) = a = \mu(a,\Unit)$ for $a\in (\NN_0)^n$,
then one can construct other biadditive binary operations on $(\NN_0)^n$, including for $n = 4$ the matrix multiplication
on $(\NN_0)^4 \cong (\NN_0)^{2\times 2}$, which is not commutative.

As we will see, this automatic associativity and commutativity of certain biadditive binary operations on commutative monoids is a rather
general effect. Results of this type have long been known, especially for multiplications on partially ordered abelian groups, where the commutative monoid
in question is the positive cone:
Vernikoff, Krein, and Tovbin \cite{vernikoff.krein.tovbin:surLesAnneauxSemiordonnes} prove that certain partially ordered real algebras
are isomorphic to a commutative algebra of real-valued functions. As the proof does not require the assumption that the
original algebra was commutative, they obtain commutativity of this algebra as a consequence of their theorem.
Kadison \cite[Sec.~3]{kadison:RepresentationTheoremForCommutativeTopologicalAlgebras} not even requires the assumption
of associativity in order to prove the same representation theorem, and thus obtains both associativity and commutativity
as consequences of the other assumptions. This result holds for certain uniformly bounded
(i.e.~the multiplicative unit is an order-unit) partially ordered real unital algebras. 
In the not necessarily unital, but lattice-ordered case, it was shown by Birkhoff and Pierce \cite[Sec.~8]{birkhoff.pierce:latticeOrderedRings}
that archimedean $f$-algebras are automatically both associative and commutative (this apparently has
been obtained independently by Amemiya \cite{amemiya:falgebrasApparently}, see e.g.~\cite{bernau.huijsman:onSomeClassesOfLatticeOrderedAlgebras}).
It is unclear whether the similarity between these two classical results has been noted before;
both of them can easily be obtained as corollaries of the following main Theorem~\ref{theorem:main}.
The crucial assumption that is made in Theorem~\ref{theorem:main} is that the operation $\mu$ is localizable,
a condition coming from real algebraic geometry
that essentially asks for compatibility with adjoining some multiplicatively inverse elements.
The proof relies on the construction and the properties of extremal
positive additive functionals on finitely generated partially ordered abelian groups.

The article is organized as follows. In the next Section~\ref{sec:prelim} we fix some definitions and notation and state
our main Theorem~\ref{theorem:main}, which is proven in Section~\ref{sec:proof}.
In Section~\ref{sec:applications} we discuss how Theorem~\ref{theorem:main} applies to partially ordered abelian groups,
recover the classical results about uniformly bounded and lattice ordered extended rings, and also obtain an apparently new application,
namely to the additive monoid $\FF^{**}$ of sums of products of squares of a division ring $\FF$:
if for all $a \in \FF$ there exists a sufficiently large natural number $k$ such that $a-k \notin \FF^{**}$,
then the multiplication of $\FF$ is commutative. The short final Section~\ref{sec:loc}
gives an intriguing observation concerning the concept of localizability.

\section{Preliminaries and Notation} \label{sec:prelim}

As discussed in the introduction, $\NN_0 = \{0,1,2, \dots\}$ are the \emph{natural numbers including $0$}, and we write
$\NN \coloneqq \NN_0 \setminus \{0\}$ for the \emph{natural numbers without $0$}. Similarly, $\ZZ$ and $\RR$ are the rings
of \emph{integers} and \emph{real numbers}.

A \emph{commutative monoid} is a set $\monoid M$ endowed with an associative and commutative operation $+$ such that there exists a
(necessarily uniquely determined) neutral element $0$ for $+$.
A submonoid of a commutative monoid $\monoid{M}$ is a subset $S$ of $\monoid M$ that is closed under addition and with $0 \in S$.
On any commutative monoid $\monoid M$ there are several,
in general non-equivalent, quasi-orders that can be defined in a natural way. We will will make use of one that is translation-invariant and unperforated:
\begin{itemize}
  \item
    A relation $\genleq$ on $\monoid M$ is \emph{translation-invariant} if for all $a,b,t \in \monoid M$, the statements
    $a \genleq b$ and $a+t \genleq b+t$ are equivalent (note that we require equivalence, not only one implication ``${\Rightarrow}$'').
  \item
    A relation $\genleq$ on $\monoid M$ is \emph{unperforated} if for all $a,b \in \monoid M$, $k\in \NN$, the statements
    $a \genleq b$ and $k a \genleq kb$ are equivalent (note that the implication ``${\Rightarrow}$'' holds trivially for translation-invariant quasi-orders).
\end{itemize}
We define the \emph{canonical quasi-order} $\canleq$ on $\monoid M$ for $a,b \in \monoid M$ as
\begin{equation}
  \label{eq:canonicalOrder}
  a \canleq b~\text{ if there exist $c, t\in \monoid M$ and $k\in \NN$ such that }~ka+c +t = kb + t
  .
\end{equation}
It is easy to check that this relation $\canleq$ is reflexive, transitive, translation-invariant, and unperforated,
and that $0 \canleq c$ for all $c\in {\monoid M}$.
Conversely, if two elements $a,b\in {\monoid M}$ fulfil $a \canleq b$, then $a\genleq b$ holds for every transitive,
translation-invariant, unperforated relation $\genleq$ on ${\monoid M}$ that fulfils $0 \genleq c$ for all $c\in {\monoid M}$: 
If $ka+c +t = kb + t$ with $c, t\in {\monoid M}$, $k\in \NN$, then $0\genleq c$ combined with translation-invariance implies
$ka+t \genleq ka+c+t = kb+t$, so $a \genleq b$ by translation-invariance and unperforatedness.

In many examples the canonical order \eqref{eq:canonicalOrder} simplifies: If ${\monoid M}$ is a convex cone
of a real vector space (i.e.~a subset with $0\in {\monoid M}$ that is closed under addition and under multiplication with positive scalars),
then $a\canleq b$ holds, for $a,b\in {\monoid M}$, if and only if $b-a \in {\monoid M}$. If ${\monoid M} = \NN_0$, then $\canleq$ is just the usual
order of natural numbers.
Note, however, that this canonical quasi-order on a general commutative monoid ${\monoid M}$ is not necessarily a partial
order because it can happen that $\canleq$ is not antisymmetric.
In particular if ${\monoid M}$ is an abelian group, i.e.~if every $a\in {\monoid M}$ has an additive inverse $-a\in {\monoid M}$, then $a\canleq b$
holds for all $a,b\in {\monoid M}$.

The canonical order $\canleq$ on a commutative monoid ${\monoid M}$ also gives rise to an equivalence relation:\
We define the equivalence relation $\canequiv$ on ${\monoid M}$ for all $a,b\in {\monoid M}$ as follows:
\begin{equation}
  a \canequiv b~\text{ if there is $d\in {\monoid M}$ such that }~\ell a \canleq \ell b + d~\textup{ and }~\ell b \canleq \ell a + d~\text{ hold for all $\ell \in \NN$.}
\end{equation}
Observe that $a \canleq b$ and $b\canleq a$ together imply that $a \canequiv b$, but not conversely.
In particular, if for two elements $a,b$ of a commutative monoid ${\monoid M}$ there exists an invertible element $c\in {\monoid M}$ such that $a+c = b$, then
$a \canleq b$ and $b \canleq a$ hold, and therefore $a \canequiv b$. In general, however, only an ``approximate'' inverse of
$c$ is required; an instructive example is:

\begin{example}
  Let ${\monoid M}$ be the convex cone ${\monoid M} \coloneqq \set{(x,y)\in \RR^2}{x>0} \cup \{(0,0)\}$ of the real vector space $\RR^2$,
  then two elements $a = (x_a,y_a) \in {\monoid M}$ and $b = (x_b,y_b) \in {\monoid M}$ fulfil $a \canequiv b$ if and only if $x_a=x_b$:
  On the one hand, if $x_a=x_b$, then $\ell a \canleq \ell b+ (1,0)$ and $\ell b \canleq \ell a+ (1,0)$ hold for all $\ell \in \NN$
  because $\ell b + (1,0)- \ell a = \ell \, (1/\ell ,y_b-y_a) \in {\monoid M}$ and $\ell a + (1,0)- \ell b = \ell \,(1/\ell ,y_a-y_b) \in {\monoid M}$. So even though
  $b-a = (0,y_b-y_a)$ and $a-b = (0,y_a-y_b)$ are not in ${\monoid M}$ in general, they can be approximated by elements of ${\monoid M}$ which is enough
  for $a \canequiv b$ to hold. On the other hand, if
  $a \canequiv b$, then there is $d = (x_d,y_d) \in {\monoid M}$ such that $\ell a \canleq \ell b + d$  and $\ell b \canleq \ell a + d$
  hold for all $\ell \in \NN$, so $b-a+ \frac{1}{\ell } d = \frac{1}{\ell }(\ell b+d-\ell a) \in {\monoid M}$ and $a-b+ \frac{1}{\ell } d = \frac{1}{\ell }(\ell a+d-\ell b) \in {\monoid M}$
  for all $\ell \in \NN$, which in particular implies $x_b-x_a + x_d/\ell  \ge 0$ and $x_a-x_b + x_d/\ell  \ge 0$ for all $\ell \in \NN_0$,
  so $x_a = x_b$.
\end{example}

We say that a map $\Phi \colon {\monoid M} \to {\monoid M}'$ between two commutative monoids ${\monoid M}$ and ${\monoid M}'$ is \emph{additive} if $\Phi(a+b) = \Phi(a)+\Phi(b)$
holds for all $a,b\in {\monoid M}$. Note that this implies that $\Phi(0)+\Phi(0) = \Phi(0)$, so $\Phi(0) \canleq 0$ and consequently
$\Phi(0) \canequiv 0$ because $0\canleq \Phi(0)$ is trivially true, but in general $\Phi(0) \neq 0$.
However, if ${\monoid{M}'}$ is an abelian group, then indeed $\Phi(0) = 0$.

In the forthcoming Lemmas~\ref{lemma:canleq} and \ref{lemma:canequiv} we will give alternative and more conceptual descriptions
of the quasi-order $\canleq$ and the equivalence relation $\canequiv$. This requires some more technical notions.

Consider any commutative monoid ${\monoid M}$, then we write $\Grot({\monoid M})$ for the \emph{Grothendieck group} of ${\monoid M}$,
i.e.~$\Grot({\monoid M}) \coloneqq ({\monoid M} \oplus {\monoid M}) / {\sim_\subGrot}$ is the quotient of the commutative monoid 
${\monoid M} \oplus {\monoid M} \coloneqq \set[\big]{(a,b)}{a,b\in {\monoid M}}$ (with elementwise addition) over the equivalence
relation $\sim_\subGrot$ that is given, for $(a,b), (a',b') \in {\monoid M} \oplus {\monoid M}$, by $(a,b) \sim_\subGrot (a',b')$ if there exists
$t\in {\monoid M}$ such that $a+b'+t = a'+b+t$; this indeed yields a well-defined abelian group. The $\sim_\subGrot$-equivalence class of
an element $(a,b) \in {\monoid M} \oplus {\monoid M}$ is denoted by $[a,b]_\subGrot \in \Grot({\monoid M})$.
Note that $-[a,b] = [b,a]$.
We also define the additive map $\iotaGrot \colon {\monoid M} \to \Grot({\monoid M})$,
\begin{equation}
  a \mapsto \iotaGrot(a) \coloneqq [a,0]_\subGrot
  ;
\end{equation}
its image $\iotaGrot(\monoid{M})$ is a submonoid of $\Grot({\monoid M})$
that generates the whole group $\Grot({\monoid M})$.

For any abelian group $\group{G}$ and any submonoid ${\monoid M}$ of $\group{G}$ we define
\begin{equation}
  \label{eq:up}
  {\monoid M}^\unperf
  \coloneqq
  \set[\big]{a\in \group{G}}{\text{there exists $k \in \NN$ such that $ka \in \monoid M$}}
  .
\end{equation}
It is easy to check that $\monoid{M}^\unperf \supseteq \monoid{M}$, that ${\monoid M}^\unperf$ is again a submonoid of $\group G$,
and that $({\monoid M}^\unperf)^\unperf = {\monoid M}^\unperf$. We say that $\monoid{M}$ is \emph{unperforated} if $\monoid{M}^\unperf = \monoid{M}$.
So ${\monoid M}^\unperf$ is the smallest unperforated submonoid of $\group{G}$ that contains $\monoid{M}$.
Note that $-({\monoid M}^\unperf) = (-{\monoid M})^\unperf$ so that we can drop the brackets in such an expression.

\begin{definition} \label{definition:GrotRedFirst}
  Let ${\monoid M}$ be a commutative monoid. Then we define the \emph{first reduced Grothendieck group} as the quotient group
  $\GrotRedFirst({\monoid M}) \coloneqq \Grot({\monoid M}) / \bigl( \iotaGrot(\monoid{M})^\unperf \cap (-\iotaGrot(\monoid{M})^\unperf) \bigr)$,
  and the map $\iotaRedFirst \colon {\monoid M} \to \GrotRedFirst({\monoid M})$,
  \begin{equation}
    a \mapsto \iotaRedFirst(a) \coloneqq [\iotaGrot(a)]_\subRedFirst
    ,
  \end{equation}
  where $[\argument]_\subRedFirst \colon \Grot({\monoid M}) \to \GrotRedFirst({\monoid M})$ is the canonical projection onto the quotient.
  We endow $\GrotRedFirst({\monoid M})$ with the order relation $\le$ that is defined, for $b,c \in \Grot({\monoid M})$,
  as $[b]_\subRedFirst \le [c]_\subRedFirst$ if $c-b \in \iotaGrot(\monoid{M})^\unperf$.
\end{definition}
The first reduced Grothendieck group $\GrotRedFirst({\monoid M})$ of a commutative monoid ${\monoid{M}}$ clearly is an abelian group,
the map $\iotaRedFirst \colon {\monoid M} \to \GrotRedFirst({\monoid M})$ is additive,
and its image $\iotaRedFirst(\monoid{M})$ is a submonoid of $\GrotRedFirst(\monoid{M})$ that generates $\GrotRedFirst(\monoid{M})$ as a group.
It is easy to check that the order relation $\le$ on $\GrotRedFirst({\monoid M})$ is well-defined and that $\le$ is an unperforated 
and translation-invariant partial order (the $\argument^\unperf$-operation guarantees unperforatedness).
In particular $\GrotRedFirst({\monoid M})$ is a partially ordered abelian group as defined in \cite{goodearl:partiallyOrderedAbelianGroupsWithInterpolation}.
Through the construction of the first reduced Grothendieck group $\GrotRedFirst({\monoid M})$ we will be able to apply some standard
techniques for partially ordered abelian groups also to the commutative monoid ${\monoid{M}}$.

\begin{lemma} \label{lemma:canleq}
  Let ${\monoid M}$ be a commutative monoid and $a,b\in \monoid{M}$, then $a \canleq b$ if and only if $\iotaRedFirst(a) \le \iotaRedFirst(b)$.
\end{lemma}
\begin{proof}
  If $a \canleq b$, i.e.~if there exist $c,t \in \monoid{M}$ and $k\in \NN$ such that $ka+c+t = kb+t$,
  then $k\bigl(\iotaGrot(b)-\iotaGrot(a)\bigr) = [kb,ka]_\subGrot = [c,0]_\subGrot \in \iotaGrot(\monoid{M})$, so
  $\iotaGrot(b)-\iotaGrot(a) \in \iotaGrot(\monoid{M})^\unperf$ and $\iotaRedFirst(a) \le \iotaRedFirst(b)$.
  Conversely, if $\iotaRedFirst(a) \le \iotaRedFirst(b)$, then then exists $k\in \NN$ such that $k\bigl(\iotaGrot(b)-\iotaGrot(a)\bigr) \in \iotaGrot(\monoid{M})$,
  i.e.~$k\bigl(\iotaGrot(b)-\iotaGrot(a)\bigr) = \iotaGrot(c)$ for some $c\in \monoid{M}$.
  This reduces to $[kb,ka]_\subGrot = [c,0]_\subGrot$, and therefore there exists $t\in \monoid{M}$ such that $kb+t = ka+c+t$.
\end{proof}

For any abelian group $\group{G}$ and any submonoid ${\monoid M}$ of $\group{G}$ we also define
\begin{equation}
  \label{eq:ddagger}
  {\monoid M}^\ddagger
  \coloneqq
  \set[\big]{a\in \group{G}}{\text{there is $e\in \group{G}$ such that $\ell a+e \in \monoid M$ for all $\ell\in \NN$}}
  .
\end{equation}
It is again easy to check that $\monoid{M}^\ddagger \supseteq \monoid{M}$ and that ${\monoid M}^\ddagger$ is a submonoid
of $\group G$. Note also that $-({\monoid M}^\ddagger) = (-{\monoid M})^\ddagger$ so that we can drop the brackets in such an expression.
In \eqref{eq:ddagger} the condition ``there is $e\in \group{G}$ such that...'' can be replaced by the formally stronger
variant ``there is $e\in \monoid{M}$ such that...'' because, if $\ell a+e \in \monoid M$ for all $\ell\in \NN$, then
$\ell a+ (a+e) \in \monoid M$ for all $\ell\in \NN$ and $a+e \in \monoid{M}$.
The special case that $\group{G}$ is a real vector space and ${\monoid M}$ a convex cone of $\group{G}$ plays an important role in
e.g.~real algebraic geometry, \cite{kuhlmann.marshall:positivitySOSandMultidimensionalMomentProblem, cimpric.marshall.netzer:closuresOfQuadraticModules}.
There the operation $\argument^\ddagger$ can be interpreted as a \emph{sequential} closure, and it has been shown that
$({\monoid M}^\ddagger)^\ddagger \supsetneq {\monoid M}^\ddagger$ in general, even in this rather well-behaved setting;
see \cite{cimpric.marshall.netzer:closuresOfQuadraticModules}.

For the combination of the operations $\argument^\unperf$ and $\argument^\ddagger$ on a submonoid $\monoid{M}$ of an abelian
group we simply write $\monoid{M}^{\unperf\ddagger} \coloneqq (\monoid{M}^\unperf)^\ddagger$.

\begin{definition}
  Let ${\monoid M}$ be a commutative monoid. Then we define the \emph{second reduced Grothendieck group} as the quotient group
  $\GrotRedSnd({\monoid M}) \coloneqq \Grot({\monoid M}) / \bigl( \iotaGrot(\monoid{M})^{\unperf\ddagger} \cap (-\iotaGrot(\monoid{M})^{\unperf\ddagger}) \bigr)$,
  and $\iotaRedSnd \colon {\monoid M} \to \GrotRedSnd({\monoid M})$,
  \begin{equation}
    a \mapsto \iotaRedSnd(a) \coloneqq [\iotaGrot(a)]_\subRedSnd
  \end{equation}
  where $[\argument]_\subRedSnd \colon \Grot({\monoid M}) \to \GrotRedSnd({\monoid M})$ is the canonical projection onto the quotient.
  We endow $\GrotRedSnd({\monoid M})$ with the order relation $\le$ that is defined, for $b,c \in \Grot({\monoid M})$,
  as $[b]_\subRedSnd \le [c]_\subRedSnd$ if $c-b \in \iotaGrot(\monoid{M})^{\unperf\ddagger}$.
\end{definition}
Like before, the second reduced Grothendieck group $\GrotRedSnd({\monoid M})$ of a commutative monoid ${\monoid{M}}$ is
a partially ordered abelian group, but its order properties will not be relevant in the following.
The map $\iotaRedSnd \colon {\monoid M} \to \GrotRedSnd({\monoid M})$ is additive,
and its image $\iotaRedSnd(\monoid{M})$ is a submonoid of $\GrotRedSnd(\monoid{M})$ that generates $\GrotRedSnd(\monoid{M})$ as a group.
The reason for introducing this second reduced Grothendieck group is that
the general result on associativity and commutativity of biadditive operations on a commutative monoid ${\monoid{M}}$ is
best formulated for $\GrotRedSnd(\monoid{M})$, this will be accomplished in Corollary~\ref{corollary:grothendieck}.

\begin{lemma} \label{lemma:canequiv}
  Let ${\monoid M}$ be a commutative monoid and $a,b\in \monoid{M}$, then $a \canequiv b$ if and only if $\iotaRedSnd(a) = \iotaRedSnd(b)$.
\end{lemma}
\begin{proof}
  First assume $a \canequiv b$, i.e.~there exists $d\in \monoid{M}$ such that $\ell a \canleq \ell b + d$ and $\ell b \canleq \ell a + d$
  hold for all $\ell \in \NN$.
  By Lemma~\ref{lemma:canleq} and by the definition of the order $\le$ on $\GrotRedFirst(\monoid{M})$
  this means that $\iotaGrot(\ell b+d)-\iotaGrot(\ell a) \in \iotaGrot(\monoid{M})^\unperf$ and 
  $\iotaGrot(\ell a+d)-\iotaGrot(\ell b) \in \iotaGrot(\monoid{M})^\unperf$ for all $\ell \in \NN$. 
  Set $e \coloneqq \iotaGrot(d) \in \Grot(\monoid{M})$, then
  $\ell \bigl(\iotaGrot(b)-\iotaGrot(a)\bigr) + e  \in \iotaGrot(\monoid{M})^\unperf$
  and $\ell \bigl(\iotaGrot(a)-\iotaGrot(b)\bigr) + e  \in \iotaGrot(\monoid{M})^\unperf$
  hold for all $\ell \in \NN$.
  This shows that $\iotaGrot(b)-\iotaGrot(a) \in \iotaGrot(\monoid{M})^{\unperf\ddagger} \cap (-\iotaGrot(\monoid{M})^{\unperf\ddagger})$,
  so $\iotaRedSnd(a) = \iotaRedSnd(b)$.
  
  Conversely, assume $\iotaRedSnd(a) = \iotaRedSnd(b)$, i.e.~$\iotaGrot(b)-\iotaGrot(a) \in \iotaGrot(\monoid{M})^{\unperf\ddagger} \cap (-\iotaGrot(\monoid{M})^{\unperf\ddagger})$,
  i.e.\ there are $e,e' \in \iotaGrot(\monoid{M})^\unperf$ such that
  $\ell \bigl(\iotaGrot(b)-\iotaGrot(a)\bigr) + e \in \iotaGrot(\monoid{M})^\unperf$ and
  $\ell \bigl(\iotaGrot(a)-\iotaGrot(b)\bigr) + e' \in \iotaGrot(\monoid{M})^\unperf$ for all $\ell, \ell' \in \NN$.
  Then by definition of $\iotaGrot(\monoid{M})^\unperf$ there are $k,k' \in \NN$ and $c,c' \in \monoid{M}$
  such that $\iotaGrot(c) = k e$ and $\iotaGrot(c') = k' e'$. Set $d \coloneqq c+c' \in \monoid{M}$,
  then adding $(k-1) e + k' e' \in \iotaGrot(\monoid{M})^\unperf$ or $ke + (k'-1)e'\in \iotaGrot(\monoid{M})^\unperf$, respectively,
  shows that
  $\ell \bigl(\iotaGrot(b)-\iotaGrot(a)\bigr) + \iotaGrot(d) \in \iotaGrot(\monoid{M})^\unperf$ and
  $\ell \bigl(\iotaGrot(a)-\iotaGrot(b)\bigr) + \iotaGrot(d) \in \iotaGrot(\monoid{M})^\unperf$ for all $\ell, \ell'$,
  i.e.~$\ell a \canleq \ell b+d$ and $\ell b \canleq \ell a+d$ for all $\ell \in \NN$ by Lemma~\ref{lemma:canleq}
  and by the definition of the order $\le$ on $\GrotRedFirst(\monoid{M})$.
  This shows that $a \canequiv b$.
\end{proof}

\begin{remark}
  We summarize the relations between these different abelian groups that can be constructed out of a commutative monoid $\monoid{M}$.
  As $\iotaGrot(\monoid{M})^\unperf \subseteq \iotaGrot(\monoid{M})^{\unperf\ddagger}$
  there is a unique group morphism $\pi_{1,2} \colon \GrotRedFirst(\monoid{M}) \to \GrotRedSnd(\monoid{M})$ that makes the triangle
  \begin{equation}
    \label{eq:triangle}
    \begin{tikzcd}
    \Grot(\monoid{M})
    \arrow{d}[swap]{[\argument]_\subRedFirst}
    \arrow{rd}{[\argument]_\subRedSnd}
    \\
    \GrotRedFirst(\monoid{M})
    \arrow{r}{\pi_{1,2}}
    &
    \GrotRedSnd(\monoid{M})
    \end{tikzcd}
  \end{equation}
  commute. Each of these abelian groups comes equipped with a submonoid: $\iotaGrot(\monoid{M}) \subseteq \Grot(\monoid{M})$, 
  $\GrotRedFirst(\monoid{M})^+ = [\iotaGrot(\monoid{M})^{\unperf}]_\subRedFirst \subseteq \GrotRedFirst(\monoid{M})$, and 
  $\GrotRedSnd(\monoid{M})^+ = [\iotaGrot(\monoid{M})^{\unperf\ddagger}]_\subRedSnd \subseteq \GrotRedSnd(\monoid{M})$;
  the three group morphisms in \eqref{eq:triangle} are surjective and can be restricted to these submonoids, i.e.~they are positive.
  These groups and their submonoids are related as follows:
  \begin{itemize}
    \item For $a,b\in \monoid{M}$ write $a \preccurlyeq_{\mathrm{alg}} b$ if there exists $c\in \monoid{M}$ such that $a+c = b$.
      Note that $\preccurlyeq_{\mathrm{alg}}$ in general is neither translation-invariant, nor unperforated, nor antisymmetric.
    \item Two elements $a,b\in \monoid{M}$ fulfil $\iotaGrot(a) = \iotaGrot(b)$ if and only if there is $t \in \monoid{M}$ such hat $a+t = b+t$. 
    \item For $a',b'\in \Grot(\monoid{M})$ write $a' \preccurlyeq_{\mathrm{Gr}} b'$ if $b'-a' \in \iotaGrot(\monoid{M})$.
      Then $a,b\in \monoid{M}$ fulfil $\iotaGrot(a) \preccurlyeq_{\mathrm{Gr}} \iotaGrot(b)$ if and only if there is $t \in \monoid{M}$
      such that $a+t \preccurlyeq_{\mathrm{alg}} b+t$. Note that $\preccurlyeq_{\mathrm{Gr}}$ is translation-invariant, but in general
      neither unperforated nor antisymmetric.
    \item Two elements $a',b'\in \Grot(\monoid{M})$ fulfil $[a']_{\subRedFirst} = [b']_{\subRedFirst}$ if and only if
      there is $k\in \NN$ such that $0 \preccurlyeq_{\mathrm{Gr}} k(a'-b') \preccurlyeq_{\mathrm{Gr}} 0$.
    \item Let $\le$ be the order of $\GrotRedFirst(\monoid{M})$, i.e.~$a'' \le b''$, for $a'',b''\in \GrotRedFirst(\monoid{M})$,
      means that $b''-a'' \in \GrotRedFirst(\monoid{M})^+$. Then $a',b'\in \Grot(\monoid{M})$ fulfil $[a']_\subRedFirst \le [b']_\subRedFirst$
      if and only if there is $k\in \NN$ such that $ka' \preccurlyeq_{\mathrm{Gr}} kb'$.
      Note that $\le$ is a translation-invariant and unperforated partial order.
    \item Two elements $a'',b'' \in \GrotRedFirst(\monoid{M})$ fulfil $\pi_{1,2}(a'') = \pi_{1,2}(b'')$ if and only if
      there is $e\in \GrotRedFirst(\monoid{M})^+$ such that $-e \le \ell(a''-b'') \le e$ for all $\ell \in \NN$.
  \end{itemize}
  The partial order on $\GrotRedSnd(\monoid{M})$ will not be relevant in the following. The relations that will be relevant
  are the partial order $\le$ on $\GrotRedFirst(\monoid{M})$, encoded on $\monoid{M}$ by the quasi-order $\canleq$ (Lemma~\ref{lemma:canleq}),
  and the equality relation $=$ on $\GrotRedSnd(\monoid{M})$, encoded on $\monoid{M}$ by the equivalence relation $\canequiv$ (Lemma~\ref{lemma:canequiv}).
\end{remark}

A \emph{binary operation} $\mu$ on a commutative monoid ${\monoid M}$ is a map $\mu \colon \monoid{M}\times \monoid{M} \to \monoid{M}$.
In this case we write $\mu^\op \colon {\monoid M} \times {\monoid M} \to {\monoid M}$,
$(a,b) \mapsto \mu^\op(a,b) \coloneqq \mu(b,a)$ for its \emph{opposite} operation.
A binary operation $\mu$ on ${\monoid{M}}$ is called \emph{biadditive} if the maps ${\monoid{M}} \ni a \mapsto \mu(a,b) \in {\monoid{M}}$
and ${\monoid{M}} \ni a \mapsto \mu(b,a) \in {\monoid{M}}$ are both additive for all $b\in \monoid{M}$.
The next definition is the key to the automatic commutativity and associativity of $\mu$:

\begin{definition} \label{definition:localizable}
  Let ${\monoid M}$ be a commutative monoid and $\mu$ a biadditive binary operation on ${\monoid M}$.
  \begin{itemize}
    \item An element $s\in {\monoid M}$ is called \emph{left $\mu$-localizable} if the following holds:\newline
      Whenever two elements $a,b\in {\monoid M}$ fulfil $\mu(s,a)+a \canleq \mu(s,b)+b$, then $a \canleq b$.
    \item An element $s\in {\monoid M}$ is called \emph{$\mu$-localizable} if $s$ is both left $\mu$-localizable and left $\mu^\op$-localizable.
    \item The commutative monoid ${\monoid M}$ is called \emph{weakly $\mu$-localizable} if for every $a\in {\monoid M}$ there exists a $\mu$-localizable element $s\in {\monoid M}$
      such that $a \canleq s$.
    \item The commutative monoid ${\monoid M}$\! is called \emph{strongly $\mu$-localizable} if every element $s$ of ${\monoid M}$ is $\mu$-localizable.
  \end{itemize}
\end{definition}
If ${\monoid M}$ is a generating semiring of a commutative real unital algebra like in \cite{schmuedgen.schoetz:positivstellensaetzForSemirings}
with multiplication $\mu \colon M \times M \to M$, then
an element $s\in {\monoid M}$ is $\mu$-localizable as defined above if and only if ${\monoid M}$
is $(1+s)$-localizable in the sense of \cite[Def.~8.5]{schmuedgen.schoetz:positivstellensaetzForSemirings}.
Weak localizability guarantees that normalized extremal positive functionals are multiplicative, which, as we will see in Section~\ref{sec:proof},
holds without any assumption of associativity or commutativity.

In Section~\ref{sec:applications} we will discuss several ways how to guarantee weak or strong localizability, and 
in Section~\ref{sec:loc} we will show that weak and strong localizability are equivalent under some additional assumptions.
However, there also are examples where localizability fails:
\begin{example}
  Consider the set
  \begin{equation*}
    {\monoid M} \coloneqq \set[\bigg]{\binom{\,a_{11}\,\,\,\,a_{12}\,}{\,a_{21}\,\,\,\,a_{22}\,}}{a_{11},a_{12},a_{21},a_{22}\in {[0,\infty[}} \subseteq \RR^{2\times 2}
  \end{equation*}
  of real $2\times 2$\,-matrices with positive (or zero) entries. Then ${\monoid M}$ with the usual addition of matrices
  is a commutative monoid, matrix multiplication restricts to a biadditive binary operation
  $\mu \colon {\monoid M} \times {\monoid M} \to {\monoid M}$, which is associative but not commutative, and the identity matrix $\Unit$ is an element of ${\monoid M}$.
  The canonical order $\canleq$ on $a,b\in {\monoid M}$ is simply given by $a \canleq b$ if and only if $b-a \in {\monoid M}$,
  because ${\monoid M}$ is a convex cone of the real vector space $\RR^{2\times 2}$. Define the matrix
  \begin{equation*}
    e \coloneqq \binom{\,0\,\,\,\,1\,}{\,1\,\,\,\,0\,} \in {\monoid M}
  \end{equation*}
  and consider any element $a \in {\monoid M}$ that fulfils $e \canleq a$,
  so in particular all entries of the matrix $a+\Unit$ are (strictly) positive.
  Either $a+\Unit$ is invertible in the ring $\RR^{2\times 2}$, then an explicit computation shows that $(a+\Unit)^{-1} \notin {\monoid M}$.
  Pick $b \in {\monoid M}$ such that $b+(a+\Unit)^{-1} \in {\monoid M}$,
  so $(a+\Unit) b \canleq (a+\Unit) b + \Unit = (a+\Unit)\bigl( b+(a+\Unit)^{-1} \bigr)$
  even though $b \not\canleq b+(a+\Unit)^{-1}$, i.e.~$a$ cannot be left $\mu$\=/localizable.
  Or $a+\Unit$ is not invertible in $\RR^{2\times 2}$, then there exists
  a matrix $c\in \RR^{2\times 2} \setminus {\monoid M}$ such that $(a+\Unit)c = 0$ (choose non-trivial elements of the kernel
  of $a+\Unit$ for the column vectors of $c$). Pick $b\in {\monoid M}$ such that $b+c \in {\monoid M}$,
  so $(a+\Unit) b = (a+\Unit)(b+c)$ even though $b \not\canleq b+c$, so $a$ again cannot be left $\mu$\=/localizable.
  This shows that ${\monoid M}$ is not weakly $\mu$\=/localizable.
\end{example}

Next we show that every biadditive binary operation on a commutative monoid lifts to the first and second reduced Grothendieck groups.
As a preparation we make the following elementary observation about its compatibility with the canonical quasi-order $\canleq$:
\begin{lemma} \label{lemma:muest}
  Let ${\monoid M}$ be a commutative monoid and $\mu$ a biadditive binary operation on ${\monoid M}$.
  If $a,a',b\in {\monoid M}$ fulfil $a \canleq a'$, then $\mu(a,b) \canleq \mu(a',b)$ and $\mu(b,a) \canleq \mu(b,a')$.
\end{lemma}
\begin{proof}
  Let $a,a',b\in {\monoid M}$ be given such that $a \canleq a'$, then there are $c,t\in {\monoid M}$ and $k\in \NN$ such that $ka+c+t = ka'+t$
  and therefore $k\mu(a,b) + \mu(c,b) + \mu(t,b) = k \mu(a',b) + \mu(t,b)$, which shows that
  $\mu(a,b) \canleq \mu(a',b)$. The same argument with $\mu^\op$ in place of $\mu$ shows that $\mu(b,a) \canleq \mu(b,a')$.
\end{proof}

\begin{proposition} \label{proposition:GrotRedFirstOperation}
  Let ${\monoid M}$ be a commutative monoid and $\mu$ a biadditive binary operation on ${\monoid M}$.
  Then there is a unique biadditive binary operation $\mu_\subRedFirst$ on $\GrotRedFirst({\monoid{M}})$
  that fulfils
  \begin{equation}
    \label{eq:GrotRedFirstOperation}
    \mu_\subRedFirst\bigl( \iotaRedFirst(a), \iotaRedFirst(b) \bigr) = \iotaRedFirst\bigl( \mu(a,b) \bigr)
    \quad\quad\text{for all $a,b \in \monoid{M}$.}
  \end{equation}
\end{proposition}
\begin{proof}
  First we show that $\mu$ lifts to a well-defined binary operation 
  $\hat \mu \colon \iotaRedFirst(\monoid{M}) \times \iotaRedFirst(\monoid{M}) \to \iotaRedFirst(\monoid{M})$,
  \begin{equation*}
    \bigl( \iotaRedFirst(a), \iotaRedFirst(b) \bigr) \mapsto \hat\mu\bigl( \iotaRedFirst(a), \iotaRedFirst(b) \bigr) \coloneqq \iotaRedFirst\bigl( \mu(a,b) \bigr)
    .
  \end{equation*}
  Consider $a,a',b,b' \in {\monoid M}$ such that $\iotaRedFirst(a) = \iotaRedFirst(a')$ and
  $\iotaRedFirst(b) = \iotaRedFirst(b')$.
  So $a \canleq a' \canleq a$ and $b \canleq b' \canleq b$ by Lemma~\ref{lemma:canleq},
  and then the chain of inequalities
  $\mu(a,b) \canleq \mu(a',b) \canleq \mu(a',b') \canleq \mu(a,b') \canleq \mu(a,b)$ holds
  by the previous Lemma~\ref{lemma:muest}.
  In particular $\mu(a,b) \canleq \mu(a',b') \canleq \mu(a,b)$.
  Applying Lemma~\ref{lemma:canleq} again shows that
  $\iotaRedFirst\bigl(\mu(a,b)\bigr) \le \iotaRedFirst\bigl(\mu(a',b')\bigr) \le \iotaRedFirst\bigl(\mu(a,b)\bigr)$,
  i.e.~$\iotaRedFirst\bigl(\mu(a,b)\bigr) = \iotaRedFirst\bigl(\mu(a',b')\bigr)$, so $\hat \mu$ is indeed well-defined.
  
  It is now easy to check that $\hat \mu$ is biadditive, and a straightforward argument shows that $\hat\mu$ can
  be extended to a biadditive binary operation $\mu_\subRedFirst$ on the whole group $\GrotRedFirst(\monoid{M})$
  because $\GrotRedFirst(\monoid{M})$ is generated by $\iotaRedFirst(\monoid{M})$.
  Conversely, as the group $\GrotRedFirst(\monoid{M})$ is generated by $\iotaRedFirst(\monoid{M})$, any biadditive binary operation on
  $\GrotRedFirst(\monoid{M})$ is uniquely determined by its restriction to $\iotaRedFirst(\monoid{M})$.
\end{proof}

\begin{proposition} \label{proposition:GrotRedSndOperation}
  Let ${\monoid M}$ be a commutative monoid and $\mu$ a biadditive binary operation on ${\monoid M}$.
  Then there is a unique biadditive binary operation $\mu_\subRedSnd$ on $\GrotRedSnd({\monoid{M}})$
  that fulfils
  \begin{equation}
    \label{eq:GrotRedSndOperation}
    \mu_\subRedSnd\bigl( \iotaRedSnd(a), \iotaRedSnd(b) \bigr) = \iotaRedSnd\bigl( \mu(a,b) \bigr)
    \quad\quad\text{for all $a,b \in \monoid{M}$.}
  \end{equation}
\end{proposition}
\begin{proof}
  Recall that the group $\GrotRedSnd(\monoid{M})$ is generated by $\iotaRedSnd(\monoid{M})$.
  Like in the proof of the previous Proposition~\ref{proposition:GrotRedFirstOperation}
  we only have to show that $\mu$ lifts to a well-defined map 
  $\hat \mu \colon \iotaRedSnd(\monoid{M}) \times \iotaRedSnd(\monoid{M}) \to \iotaRedSnd(\monoid{M})$,
  \begin{equation*}
    \bigl( \iotaRedSnd(a), \iotaRedSnd(b) \bigr) \mapsto \hat\mu\bigl( \iotaRedSnd(a), \iotaRedSnd(b) \bigr) \coloneqq \iotaRedSnd\bigl( \mu(a,b) \bigr)
    .
  \end{equation*}
  So consider $a,a',b,b' \in {\monoid M}$ such that $\iotaRedSnd(a) = \iotaRedSnd(a')$ and
  $\iotaRedSnd(b) = \iotaRedSnd(b')$, i.e.~$a\canequiv a'$ and $b \canequiv b'$ by Lemma~\ref{lemma:canequiv}.
  This means there are $c,d \in {\monoid M}$ such that
  $ka \canleq ka' + c$, $ka' \canleq ka + c$, $kb \canleq kb' + d$, and $kb' \canleq kb + d$
  hold for all $k\in \NN$. By Lemma~\ref{lemma:muest}
  \begin{align*}
    k\mu(a,b) &\canleq k \mu(a',b) + \mu(c,b) \canleq k\mu(a',b') + \mu(a',d) + \mu(c,b)
  \shortintertext{and}
    k\mu(a',b') &\canleq k \mu(a',b) + \mu(a',d) \canleq k\mu(a,b) + \mu(c,b) + \mu(a',d)
  \end{align*}
  hold for all $k\in\NN$. This shows that $\mu(a,b) \canequiv \mu(a',b')$ and therefore $\iotaRedSnd\bigl( \mu(a,b) \bigr) = \iotaRedSnd\bigl( \mu(a',b') \bigr)$
  by applying Lemma~\ref{lemma:canequiv} again.  
\end{proof}

We now state our main theorem and two immediate corollaries. The proof of the theorem will be given in the next Section~\ref{sec:proof}.

\newpage

\begin{theorem} \label{theorem:main}
  Let ${\monoid M}$ be a commutative monoid, $\mu$ a biadditive binary operation on ${\monoid M}$,
  and assume that ${\monoid M}$ is weakly $\mu$\=/localizable. Then the relations
  \begin{equation}
    \mu(a,b) \canequiv \mu(b,a)
    \quad\quad\text{and}\quad\quad
    \mu\bigl(\mu(a,b),c\bigr) \canequiv \mu\bigl(a,\mu(b,c)\bigr)
  \end{equation}
  hold for all $a,b,c\in {\monoid M}$.
\end{theorem}

\begin{corollary} \label{corollary:equivIsEqual}
  Let ${\monoid M}$ be a commutative monoid, $\mu$ a biadditive binary operation on $\monoid{M}$, and assume that ${\monoid M}$ is weakly $\mu$\=/localizable.
  If the equivalence relation $\canequiv$ on ${\monoid M}$ is the equality relation $=$, then the operation $\mu$ on ${\monoid M}$ is both associative and commutative.
\end{corollary}

\begin{corollary} \label{corollary:grothendieck}
  Let ${\monoid M}$ be a commutative monoid and $\mu$ a biadditive binary operation on $\monoid{M}$.
  Consider the biadditive binary operation $\mu_\subRedSnd$ on the second reduced Grothendieck group $\GrotRedSnd({\monoid M})$
  that was constructed in Proposition~\ref{proposition:GrotRedSndOperation}.
  If ${\monoid M}$ is weakly $\mu$\=/localizable, then $\mu_\subRedSnd$ is both associative and commutative,
  i.e.~$\GrotRedSnd({\monoid M})$ with multiplication $\mu_\subRedSnd$ is a commutative ring (possibly without unit).
\end{corollary}
\begin{proof}
  Apply Theorem~\ref{theorem:main} and Lemma~\ref{lemma:canequiv} and use that $\GrotRedSnd({\monoid M})$
  is generated by $\iotaRedSnd(\monoid{M})$.
\end{proof}

\section{Proof of the main theorem} \label{sec:proof}
The proof of Theorem~\ref{theorem:main} exploits the properties of extremal positive additive functionals on 
partially ordered abelian groups:

A \emph{partially ordered abelian group} is an abelian group ${\group G}$ endowed with a translation-invariant partial order $\le$.
Its positive cone $\group{G}^+ \coloneqq \set[\big]{a\in \group{G}}{0 \le a}$ then is a submonoid of $\group{G}$.
In particular, the first and second reduced Grothendieck groups of any commutative monoid $\monoid{M}$ are
partially ordered abelian groups, and they fulfil $\GrotRedFirst(\monoid{M})^+ \supseteq \iotaRedFirst(\monoid{M})$ and
$\GrotRedSnd(\monoid{M})^+ \supseteq \iotaRedSnd(\monoid{M})$.
For any group $\group G$ and any finite subset $F$ of $\group G$ we write
\begin{equation}
  \genGrp{F} \coloneqq \set[\Big]{\sum\nolimits_{f \in F} k_f f}{k_f \in \ZZ\text{ for all }f\in F}
\end{equation}
for the subgroup of $\group G$ that is generated by $F$.
If $\group{G}$ is a partially ordered abelian group and $\group{H}$ any subgroup of $\group G$, then $\group{H}$ with the 
restriction of the order of $\group G$ is again a partially ordered abelian group with positive cone
$\group{H}^+ = \group{H}\cap \group{G}^+$.

An \emph{additive functional} on an abelian group $\group G$ is an additive map (i.e.~a group homomorphism) ${\group G} \to \RR$.
The set of all additive functionals on ${\group G}$ is a linear subspace of the real vector space of all maps ${\group G}\to \RR$ (with pointwise operations).
The restriction of an additive functional $\omega$ on an abelian group $\group{G}$ to a subgroup ${\group H}$ of $\group{G}$ is denoted by
$\omega\at{{\group H}} \colon {\group H} \to \RR$; $h\mapsto \omega\at{{\group H}}(h) \coloneqq \omega(h)$.
An additive functional $\omega$ on a partially ordered abelian group ${\group G}$ is called \emph{positive} if $\omega(a) \le \omega(b)$ holds
for all $a,b\in {\group G}$ with $a \le b$. Equivalently, $\omega$ is positive if and only if $\omega(a) \ge 0$ for all $a\in \group{G}^+$.
The set of positive additive functionals on ${\group G}$ is a 
convex cone of the real vector space of all additive functionals on ${\group G}$.
A positive additive functional $\omega$ on ${\group G}$ is called \emph{extremal}
if the following is fulfilled: Whenever two positive additive functionals $\sigma_1,\sigma_2$ on ${\group G}$ fulfil $\sigma_1 + \sigma_2 = \omega$,
there exists $\lambda \in {[0,1]}$ such that $\sigma_1 = \lambda \omega$ and $\sigma_2 = (1-\lambda)\omega$.
If $\omega$ is an extremal positive additive functional on ${\group G}$, then $\lambda \omega$ also is extremal for all $\lambda \in {]0,\infty[}$.

In order to prove our main Theorem~\ref{theorem:main} we need two basic results about (extremal) positive additive functionals
on partially ordered abelian groups, more precisely on finitely generated subgroups of the first reduced Grothendieck group
of a commutative monoid $\monoid{M}$. These two results are cited from the literature in the forthcoming Lemmas~\ref{lemma:extension}
and Lemma~\ref{lemma:positivstellensatz}. As we only need these results for the finitely generated case, their
proofs do not require to assume the axiom of choice.

\begin{lemma} \label{lemma:extension}
  Let ${\group G}$ be a partially ordered abelian group, $F$ a finite subset of ${\group G}^+$, and ${\widecheck{\group H}}$ a subgroup of $\genGrp{F}$.
  Assume there is an element $e \in {\widecheck{\group H}}$ such that $f \le e$ for all $f\in F$.
  Let $\check\omega$ be a positive additive functional on $\widecheck{\group H}$, then there exists a positive additive functional $\omega$
  on $\genGrp{F}$ such that $\omega\at{{\widecheck{\group H}}} = \check\omega$.
\end{lemma}
\begin{proof}
  Consider a general element $g = \sum_{f\in F} k_f f \in \genGrp{F}$ with $k_f \in \ZZ$ for all $f\in F$,
  then $g \le \sum_{f\in F} \abs{k_f} f \le \bigl( \sum_{f\in F} \abs{k_f} \bigr) e \in {\widecheck{\group H}}$.
  By \cite[Prop.~4.2]{goodearl:partiallyOrderedAbelianGroupsWithInterpolation}, applied to the partially ordered
  abelian group $\genGrp{F}$ and the positive additive functional $\check\omega$ on its subgroup ${\widecheck{\group H}}$, there exists a
  positive additive functional $\omega$ on $\genGrp{F}$ such that $\omega\at{{\widecheck{\group H}}} = \check \omega$.
\end{proof}
The following technical proposition and its corollary are our most crucial intermediary results:

\begin{proposition} \label{proposition:mult}
  Let ${\group G}$ be a partially ordered abelian group and $\mu$ a biadditive binary operation
  on $\group{G}$ such that $\mu(a,b)\in \group{G}^+$ for all $a,b\in \group{G}^+$.
  Let $F$ be a finite subset of ${\group G}^+$ with a largest element $s\in F$ such that the following left localizability condition is fulfilled:
  \begin{equation}
    \label{eq:mult:localizability}
    \text{Whenever $a \in \group{G}$ fulfils $\mu(s,a) + a \in \group{G}^+$, then $a \in \group{G}^+$.}
  \end{equation}
  Let $\varphi$ be an extremal positive additive functional on the subgroup
  \begin{equation}
    {\group H}
    \coloneqq
    \genGrp[\big]{ F \cup \set[\big]{\mu(f,f')} {f,f' \in F} }
  \end{equation}
  of $\group{G}$ such that $\varphi(s) > 0$ and $\varphi(\mu(s,s)) > 0$. Then
  \begin{equation}
    \label{eq:proposition:mult}
    \varphi(s) \,\varphi\bigl(\mu(f,f')\bigr) = \varphi\bigl(\mu(f,s)\bigr)\,\varphi(f')
    \quad\quad\text{for all $f,f' \in F$.}
  \end{equation}
\end{proposition}
\begin{proof}
  Note that the estimate
  $\mu(f,f') \le \mu(f,s) \le \mu(s,s)$ holds for all $f,f' \in F$, and analogously also
  $\mu\bigl(f,\mu(f',f'')\bigr) \le \mu\bigl(s,\mu(s,s)\bigr)$ for all $f,f',f''\in F$.
  Define
  \begin{equation*}
    {\group H}_\ex \coloneqq \genGrp[\big]{ F \cup \set[\big]{\mu(f,f')} {f,f' \in F} \cup \set[\big]{\mu(f,\mu(f',f''))} {f,f',f'' \in F} }
    .
  \end{equation*}
  
  The set $\widecheck{\group H}_\ex \coloneqq \set{\mu(s,h)+h}{h\in {\group H}}$ clearly is a subgroup of ${\group H}_\ex$.
  If $h \in {\group H}$ fulfils $\mu(s,h)+h \in \group{G}^+$, then $h \in \group{G}^+$ by the left localizability assumption
  \eqref{eq:mult:localizability}, so $\varphi(h) \ge 0$ by positivity of $\varphi$.
  In particular if $\mu(s,h)+h = 0 \in \group{G}^+ \cap (-\group{G}^+)$, then $\varphi(h) = 0$.
  It follows that the map $\check \rho_1 \colon \widecheck{\group H}_\ex \to \RR$,
  \begin{equation*}
    \mu(s,h)+h \mapsto \check\rho_1\bigl(\mu(s,h)+h\bigr) \coloneqq \varphi(h)
  \end{equation*}
  is a well-defined positive additive functional.
  Set $e \coloneqq \mu\bigl(s,\mu(s,s)\bigr)+ 2\mu(s,s) + s = \mu(s,h_e) + h_e$ with $h_e \coloneqq \mu(s,s) + s \in {\group H}$,
  then $e \in \widecheck{\group H}_\ex$ and the estimate $g \le e$ holds for all
  $g \in F \cup \set[\big]{\mu(f,f')} {f,f' \in F} \cup \set[\big]{\mu\bigl(f,\mu(f',f'')\bigr)} {f,f',f'' \in F}$
  because $f \le s$, $\mu(f,f') \le \mu(s,s)$, and $\mu\bigl(f,\mu(f',f'')\bigr) \le \mu\bigl(s,\mu(s,s)\bigr)$
  for $f,f',f''\in F$.
  The previous Lemma~\ref{lemma:extension} now provides a positive additive functional
  $\rho_1 \colon {\group H}_\ex \to \RR$ such that $\rho_1(\mu(s,h)+h) = \check \rho_1(\mu(s,h)+h) = \varphi(h)$ for all $h\in {\group H}$.
      
  Consider the positive additive functional $\rho_2 \coloneqq {\group H} \to \RR$, $h \mapsto \rho_2(h) \coloneqq \rho_1(\mu(s,h))$.
  For all $h\in {\group H}$ the identity $\rho_1(h) + \rho_2(h) = \rho_1(\mu(s,h)+h) = \varphi(h)$ holds,
  i.e.~$\rho_1\at{{\group H}} + \rho_2 = \varphi$.
  By extremality of $\varphi$ there exists $\lambda\in [0,1]$ such that
  $\rho_1\at{{\group H}} = \lambda \varphi$. Note that $\lambda \neq 0$ because $0 < \varphi(s) = \rho_1\at{{\group H}}\bigl(\mu(s,s)+s\bigr) = \lambda \varphi\bigl(\mu(s,s)+s\bigr)$.
  
  For any $f\in F$ we now define the two positive additive functionals $\sigma_{f;1}, \sigma_{f;2} \colon {\group H} \to \RR$,
  \begin{align*}
    h &\mapsto \sigma_{f;1}(h) \coloneqq \rho_1\bigl(\mu(f,h)+h\bigr)
  \shortintertext{and}
    h &\mapsto \sigma_{f;2}(h) \coloneqq \rho_1\bigl(\mu(s-f,h)\bigr)
    .    
  \end{align*}
  Clearly $\sigma_{f;1} + \sigma_{f;2} = \varphi$, so by extremality of $\varphi$ there exists $\lambda_f \in {[0,1]}$
  such that $\sigma_{f;1} = \lambda_f \varphi$.
  For all $f,f' \in F$ it follows that $\lambda_f \varphi(f') = \sigma_{f;1}(f') = \rho_1\bigl(\mu(f,f')+f'\bigr) = \lambda \varphi\bigl(\mu(f,f')+f'\bigr)$,
  where we use that $\mu(f,f')+f' \in {\group H}$ by definition of ${\group H}$. Reordering yields
  $\varphi(\mu(f,f')) = \bigl((\lambda_f/\lambda)-1\bigr) \varphi(f')$, and the special case $f' \coloneqq s$ shows that $(\lambda_f/\lambda)-1 = \varphi(s)^{-1}\varphi(\mu(f,s))$,
  so \eqref{eq:proposition:mult} holds.
\end{proof}

\begin{remark}
  The proof of Proposition~\ref{proposition:mult} consists of two steps, which are motivated by two different arguments to prove
  that extremal states on certain ordered algebras are multiplicative, see
  \cite[Thm.~4.20]{schoetz:gelfandNaimarkTheorems} and \cite[Thm.~2]{bucy.maltese:RepresentationTheoremForPositiveFunctionals}:
  The first step is the construction of the positive additive functional
  $\rho_1 \at{{\group H}} \colon {\group H} \to \RR$ that fulfils $\rho_1 \at{{\group H}} = \lambda \varphi$. By evaluation
  on $\mu(s,f)+f \in {\group H}$ one can show that $\varphi(f) = \lambda \varphi\bigl(\mu(s,f)+f\bigr)$, so $\varphi\bigl(\mu(s,f)\bigr) = (1/\lambda-1)\varphi(f)$
  for all $f\in F$. In particular for $f \coloneqq s$ this shows that $(1/\lambda-1) = \varphi\bigl(\mu(s,s)\bigr) / \varphi(s)$,
  and therefore $\varphi(s) \varphi\bigl(\mu(s,f)\bigr) = \varphi\bigl(\mu(s,s)\bigr) \varphi(f)$,
  which is \eqref{eq:proposition:mult} in the special case $f=s$.
  In the second step, a similar argument is then used in order to obtain the general form of \eqref{eq:proposition:mult}.
\end{remark}

\begin{corollary} \label{corollary:mult}
  Let ${\group G}$ be a partially ordered abelian group and $\mu$ a biadditive binary operation
  on $\group{G}$ such that $\mu(a,b)\in \group{G}^+$ for all $a,b\in \group{G}^+$.
  Let $F$ be a finite subset of ${\group G}^+$ with a largest element $s\in F$ such that the following localizability condition is fulfilled:
  \begin{equation}
    \label{eq:multCor}
    \text{Whenever $a \in \group{G}$ fulfils $\mu(s,a) + a \in \group{G}^+$ or $\mu(a,s) + a \in \group{G}^+$, then $a \in \group{G}^+$.}
  \end{equation}
  Like in the previous Proposition~\ref{proposition:mult} write
  \begin{equation}
    {\group H}
    \coloneqq
    \genGrp[\big]{ F \cup \set[\big]{\mu(f,f')} {f,f' \in F} }
    .
  \end{equation}
  We call an additive functional $\psi$ on ${\group H}$ \emph{multiplicative} if
  $\psi\bigl(\mu(f,f')\bigr) = \psi(f)\,\psi(f')$ for all $f,f' \in F$.
  Assume $g,g' \in F \cap \set{\mu(f,f')}{f,f'\in F}$ fulfil $\psi(g) = \psi(g')$ for all multiplicative additive functionals
  $\psi$ on ${\group H}$. Then $\varphi(g) = \varphi(g')$ for all extremal positive additive functionals $\varphi$ on ${\group H}$.
\end{corollary}
\begin{proof}
  Consider any extremal positive additive functional $\varphi$ on ${\group H}$,
  then $\varphi(s) \ge 0$ and $\varphi(\mu(s,s)) \ge 0$.
  If $\varphi(s) = 0$, then $0 \le \varphi(f) \le \varphi(s) = 0$ for all $f\in F$,
  in particular $\varphi(g) = 0 = \varphi(g')$.
  If $\varphi(\mu(s,s)) = 0$, then $0 \le \varphi(\mu(f,f')) \le \varphi(\mu(f,s)) \le \varphi(\mu(s,s)) = 0$ for all $f,f' \in F$,
  and again $\varphi(g) = 0 = \varphi(g')$.
  It only remains to treat the case that $\varphi(s) > 0$ and $\varphi(\mu(s,s)) > 0$:
  In this case Proposition~\ref{proposition:mult} applies to the biadditive binary operations $\mu$ and $\mu^\op$ on $\group{G}$
  and shows that
  \begin{equation*}
    \varphi(s) \,\varphi\bigl(\mu(f,f')\bigr) = \varphi\bigl(\mu(f,s)\bigr)\,\varphi(f')
    \quad\quad\text{and}\quad\quad
    \varphi(s) \,\varphi\bigl(\mu(f,f')\bigr) = \varphi\bigl(\mu(s,f')\bigr)\,\varphi(f)
  \end{equation*}
  hold for all $f,f'\in F$.
  Multiplying the first identity with $\varphi(s)$ and then applying the second one in the special case
  $f' \coloneqq s$ yields
  \begin{equation*}
    \varphi(s)^2 \,\varphi\bigl(\mu(f,f')\bigr) = \varphi(s)\,\varphi\bigl(\mu(f,s)\bigr)\,\varphi(f') = \varphi\bigl(\mu(s,s)\bigr)\,\varphi(f)\,\varphi(f')
    .
  \end{equation*}
  It follows that the additive functional
  $\psi \coloneqq \varphi(\mu(s,s)) \varphi(s)^{-2} \varphi \colon {\group H} \to \RR$ is multiplicative,
  so $\psi(g) = \psi(g')$ by assumption, and therefore $\varphi(g) = \varphi(g')$.
\end{proof}
Corollary~\ref{corollary:mult} is the first important ingredient in the proof of Theorem~\ref{theorem:main}.

\begin{lemma} \label{lemma:positivstellensatz}
  Let ${\group G}$ be a partially ordered abelian group, $F$ a finite subset of ${\group G}^+\!$, and $a,b \in F$.
  Assume that all non-zero extremal positive additive functionals $\varphi \colon \genGrp{F} \to \RR$ fulfil
  $\varphi(a) < \varphi(b)$. Then there exists $k\in \NN$ such that $ka \le kb$.
\end{lemma}
\begin{proof}
  Set $e \coloneqq \sum_{f\in F} f \in \group{G}^+$.
  A general element $g = \sum_{f\in F} k_f f \in \genGrp{F}$ with $k_f \in \ZZ$ for $f\in F$
  fulfils $g \le \sum_{f\in F} \abs{k_f} f \le \bigl( \sum_{f\in F} \abs{k_f} \bigr) e$ (in other words, $e$ is an order-unit of $\genGrp{F}$).
  Write $C$ for the set of all positive additive functionals $\omega$ on $\genGrp{F}$ that fulfil $\omega(e) = 1$ (i.e., $\omega$ is a state
  of $(\genGrp{F}, e)$; in the notation of \cite{goodearl:partiallyOrderedAbelianGroupsWithInterpolation}, $\omega \in S(\genGrp{F}, e)$).
  Clearly $C$ is a convex subset of the real vector space of all additive functionals on $\genGrp{F}$.
  Consider any extreme point $\omega$ of $C$, i.e.~the only elements $\rho_1,\rho_2 \in C$, for which there exists $\lambda \in {]0,1[}$
  such that $\omega = \lambda \rho_1 + (1-\lambda) \rho_2$, are $\rho_1 = \rho_2 = \omega$.
  Then $\omega$ is a non-zero extremal positive additive functional on $\genGrp{F}$ and $0 < \omega(b-a)$ by assumption:
  
  Indeed, assume that there are two positive 
  additive functionals $\sigma_1$ and $\sigma_2$ on $\genGrp{F}$ such that $\sigma_1 + \sigma_2 = \omega$.
  Set $\lambda \coloneqq \sigma_1(e)$ and note that $1-\lambda = \omega(e)-\sigma_1(e) = \sigma_2(e)$.
  If $\lambda = 0$, then $\sigma_1(e) = 0$, which implies $0 \le \sigma_1(f) \le \sigma_1(e) = 0$
  for all $f\in F$, i.e.~$\sigma_1 = 0 = \lambda \omega$ and therefore $\sigma_2 = \omega = (1-\lambda)\omega$.
  If $\lambda = 1$, then $\sigma_2(e) = \omega(e)-\sigma_1(e) = 0$
  and therefore $\sigma_2 = 0 = (1-\lambda)\omega$ and $\sigma_1 = \omega = \lambda \omega$ by a similar argument.
  Finally, if $\lambda \in {]0,1[}$,
  set $\rho_1 \coloneqq \lambda^{-1} \sigma_1 \in C$ and $\rho_2 \coloneqq (1-\lambda)^{-1} \sigma_2 \in C$,
  then $\lambda \rho_1 + (1-\lambda) \rho_2 = \sigma_1+\sigma_2 = \omega$ and therefore $\rho_1 = \rho_2 = \omega$ because
  $\omega$ is an extreme point of $C$, so again $\sigma_1 = \lambda \omega$ and $\sigma_2 = (1-\lambda)\omega$.
  
  We have thus shown that $0 < \omega(b-a)$ for all extreme points $\omega$ of $C$. So
  \cite[Thm.~6.3]{goodearl:partiallyOrderedAbelianGroupsWithInterpolation} shows that
  there is $k\in \NN$ such that $0 \le k(b-a)$, i.e.~$ka \le kb$.
\end{proof}

\begin{proposition} \label{proposition:asymp}
  Let ${\monoid M}$ be a commutative monoid, $F$ a finite subset of ${\monoid M}$, and $g,g' \in F$.
  Consider the subgroup $\genGrp{\iotaRedFirst(F)}$ of the first reduced Grothendieck group $\GrotRedFirst(\monoid{M})$
  and assume that all extremal positive additive functionals $\varphi$ on $\genGrp{\iotaRedFirst(F)}$ fulfil
  $\varphi(\iotaRedFirst(g)) = \varphi(\iotaRedFirst(g'))$. Then $g \canequiv g'$.
\end{proposition}
\begin{proof}
  Set $e \coloneqq \sum_{f\in F} f$ and consider any non-zero extremal positive additive functional $\varphi$ on $\genGrp{\iotaRedFirst(F)}$.
  Then there exists $f\in F$ such that $0 < \varphi\bigl(\iotaRedFirst(f)\bigr)$, 
  and therefore also $0 < \varphi\bigl(\iotaRedFirst(e)\bigr)$. It follows that
  $\varphi\bigl(\iotaRedFirst(\ell g)\bigr) < \varphi\bigl(\iotaRedFirst(\ell g' + e)\bigr)$ 
  and 
  $\varphi\bigl(\iotaRedFirst(\ell g')\bigr) < \varphi\bigl(\iotaRedFirst(\ell g + e)\bigr)$
  hold for all $\ell \in \NN$ and all non-zero extremal positive additive functionals $\varphi$ on $\genGrp{\iotaRedFirst(F)}$.
  The previous Lemma~\ref{lemma:positivstellensatz} now shows that there are $k,k' \in \NN$ such that
  $k\,\iotaRedFirst(\ell g) \le k\,\iotaRedFirst(\ell g' + e)$ and $k'\iotaRedFirst(\ell g') \le k'\iotaRedFirst(\ell g + e)$
  hold for all $\ell \in \NN$.
  By Lemma~\ref{lemma:canleq} this means that $k \ell g \canleq k(\ell g' + e)$ and $k'\ell g' \canleq k'(\ell g + e)$ for all $\ell \in \NN$.
  As the canonical quasi-order $\canleq$ on $\monoid{M}$ is unperforated by construction, $\ell g \canleq \ell g' + e$ and $\ell g' \canleq \ell g + e$
  hold for all $\ell \in \NN$, i.e.~$g \canequiv g'$.
\end{proof}
This Proposition~\ref{proposition:asymp} is the second ingredient in the proof of Theorem~\ref{theorem:main}:

\begin{proof}[of Theorem~\ref{theorem:main}]
  Let ${\monoid M}$ be a commutative monoid, $\mu \colon {\monoid M} \times {\monoid M} \to {\monoid M}$ a biadditive binary operation,
  and assume that ${\monoid M}$ is weakly $\mu$\=/localizable. The first reduced Grothendieck group $\GrotRedFirst(\monoid{M})$
  is a partially ordered abelian group and by Proposition~\ref{proposition:GrotRedFirstOperation} there exists a unique biadditive binary operation $\mu_\subRedFirst$
  on $\GrotRedFirst(\monoid{M})$ such that $\mu_\subRedFirst\bigl( \iotaRedFirst(a), \iotaRedFirst(b) \bigr) = \iotaRedFirst\bigl( \mu(a,b) \bigr)$
  for all $a,b\in \monoid{M}$. If $c\in \GrotRedFirst(\monoid{M})^+$\!, then by definition of $\GrotRedFirst(\monoid{M})^+$ there exists $k\in \NN$
  and $a \in \monoid{M}$ such that $k c = \iotaRedFirst(a)$, see \eqref{eq:up} and Definition~\ref{definition:GrotRedFirst}.
  So for $c,d\in \GrotRedFirst(\monoid{M})^+$ there are $k,k'\in \NN$ and $a, b \in \monoid{M}$ such that
  $k c = \iotaRedFirst(a)$ and $k' d = \iotaRedFirst(b)$, and therefore
  $kk' \mu_\subRedFirst(c,d) = \mu_\subRedFirst\bigl( \iotaRedFirst(a), \iotaRedFirst(b) \bigr) = \iotaRedFirst\bigl( \mu(a,b)\bigr)$,
  which shows that $\mu_\subRedFirst(c,d) \in \GrotRedFirst(\monoid{M})^+$\!.
  
  Consider any finite subset $F'$ of $\monoid{M}$. By weak localizability of $\monoid{M}$ there exists a $\mu$-localizable element $s \in \monoid{M}$
  such that $\sum_{f\in F'} f\canleq s$, in particular $f\canleq s$ for all $f\in F'$. Set $F \coloneqq F' \cup \{s\}$, then $\iotaRedFirst(f) \le \iotaRedFirst(s)$
  holds for all $f\in F$ by Lemma~\ref{lemma:canleq}. We show that $\iotaRedFirst(s)$ fulfils the localizability condition
  \eqref{eq:multCor}. Indeed, let $c \in \GrotRedFirst(\monoid{M})$; write $c = \iotaRedFirst(b) - \iotaRedFirst(a)$ with $a,b\in \monoid{M}$.
  If $\mu_\subRedFirst\bigl(\iotaRedFirst(s),c\bigr) + \iotaRedFirst(c) \in \GrotRedFirst(\monoid{M})^+\!$, then
  \begin{equation*}
    \iotaRedFirst\bigl( \mu(s,a) + a \bigr)
    =
    \mu_\subRedFirst\bigl(\iotaRedFirst(s), \iotaRedFirst(a)\bigr) + \iotaRedFirst(a)
    \le
    \mu_\subRedFirst\bigl(\iotaRedFirst(s), \iotaRedFirst(b)\bigr) + \iotaRedFirst(b)
    =
    \iotaRedFirst\bigl( \mu(s,b) + b \bigr)
  \end{equation*}
  holds, so $\mu(s,a) + a \canleq \mu(s,b) + b$ by Lemma~\ref{lemma:canleq}. Similarly, if 
  $\mu_\subRedFirst\bigl(c,\iotaRedFirst(s)\bigr) + \iotaRedFirst(c) \in \GrotRedFirst(\monoid{M})^+\!$,
  then $\mu(a,s) + a \canleq \mu(b,s) + b$. Therefore, $\mu$\=/localizability of $s$ implies $a \canleq b$,
  so $\iotaRedFirst(a) \le \iotaRedFirst(b)$ by Lemma~\ref{lemma:canleq}, i.e.~$0 \le c$ as required in \eqref{eq:multCor}.
  
  Let us apply Corollary~\ref{corollary:mult} to $\GrotRedFirst(\monoid{M})$, $\mu_\subRedFirst$, and the finite subset
  $\iotaRedFirst(F)$ of $\GrotRedFirst(\monoid{M})^+\!$. Write
  \begin{equation*}
    \group{H}
    \coloneqq
    \genGrp[\big]{\iotaRedFirst(F) \cup \set[\big]{\mu_\subRedFirst\bigl(\iotaRedFirst(f),\iotaRedFirst(f')\bigr)}{f,f' \in F}}
    =
    \genGrp[\big]{\iotaRedFirst\bigl( F \cup \set{\mu(f,f')}{f,f' \in F}\bigr)}.
  \end{equation*}
  Then an additive functional $\psi$ on $\group{H}$ is multiplicative as defined in Corollary~\ref{corollary:mult} if and only if
  \begin{equation*}
    \psi\bigl( \iotaRedFirst(\mu(f,f'))\bigr) = \psi\bigl( \iotaRedFirst(f)\bigr) \psi\bigl( \iotaRedFirst(f')\bigr)
    \quad\quad\text{for all $f,f' \in F$.}
  \end{equation*}
  If $g,g' \in F' \cap \set{\mu(f,f')}{f,f' \in F'} \subseteq F \cap \set{\mu(f,f')}{f,f' \in F}$
  fulfil $\psi(\iotaRedFirst(g)) = \psi(\iotaRedFirst(g'))$
  for all multiplicative additive functionals $\psi$ on $\group{H}$,
  then Corollary \ref{corollary:mult} shows that
  $\varphi(\iotaRedFirst(g)) = \varphi(\iotaRedFirst(g'))$
  for all extremal positive additive functionals $\varphi$ on $\group{H}$,
  and therefore $g \canequiv g'$ by the previous Proposition~\ref{proposition:asymp}
  applied to the finite subset $F \cup \set{ \mu(f,f')}{f,f' \in F}$
  of $\monoid{M}$.
  
  In particular, given $a,b \in {\monoid M}$, then by choosing
  $F' \coloneqq \big\{ a, b, \mu(a,b), \mu(b,a) \big\}$ we obtain $\mu(a,b) \canequiv \mu(b,a)$
  because every multiplicative additive functional $\psi$ fulfils
  \begin{equation*}
    \psi\bigl( \iotaRedFirst(\mu(a,b))\bigr)
    = 
    \psi\bigl( \iotaRedFirst(a)\bigr) \psi\bigl( \iotaRedFirst(b)\bigr)
    =
    \psi\bigl( \iotaRedFirst(\mu(b,a))\bigr)
    .
  \end{equation*}
  Similarly, given $a,b,c \in {\monoid M}$, then by choosing
  $F' \coloneqq \big\{  a,b,c,\mu(a,b) ,\mu(b,c), \mu(\mu(a,b),c), \mu(a,\mu(b,c)) \big\}$ we obtain $\mu(\mu(a,b),c) \canequiv \mu(a,\mu(b,c))$
  because every multiplicative additive functional $\psi$ fulfils 
  \begin{equation*}
    \psi\bigl( \iotaRedFirst\bigl(\mu(\mu(a,b),c)\bigr)\bigr) = 
    \psi\bigl( \iotaRedFirst(a)\bigr) \psi\bigl( \iotaRedFirst(b)\bigr) \psi\bigl( \iotaRedFirst(c)\bigr) =
    \psi\bigl( \iotaRedFirst\bigl(\mu(a,\mu(b,c))\bigr)\bigr)
    .
  \end{equation*}
\end{proof}

\section{Applications} \label{sec:applications}

Theorem~\ref{theorem:main} can be applied in particular to the positive cone $\group{G}^+$ of a partially ordered abelian group $\group{G}$.
The result for $\group{G}^+$ then extends to the whole $\group{G}$ if $\group{G}$ is \emph{directed}, i.e.~if for every element $a\in \group{G}$ there
are $b,c\in\group{G}^+$ such that $a = b-c$. In most cases, the canonical quasi-order on $\group{G}^+$ is induced by the given order of $\group{G}$:

\begin{proposition} \label{proposition:orderedGrpCanleq}
  Let $\group{G}$ be a partially ordered abelian group and assume that its order relation $\le$ is unperforated. Then
  the canonical quasi-order $\canleq$ on $\group{G}^+$ is the restriction of $\le$ from $\group{G}$ to $\group{G}^+$.
\end{proposition}
\begin{proof}
  Let $a,b\in \group{G}^+$ be given. If $a \canleq b$, then by definition there are $c,t\in \group{G}^+$
  and $k\in \NN$ such that $ka + c + t = kb + t$. So the estimate $ka + t \le ka + c + t = kb + t$ holds,
  and therefore $a\le b$ by translation-invariance and unperforatedness of the order relation $\le$. Conversely,
  if $a,b\in \group{G}^+$ fulfil $a \le b$, then $0\le b-a$, so $a+c = b$ with $c \coloneqq b-a \in \group{G}^+$,
  which shows that $a \canleq b$.
\end{proof}

There are several different ways to guarantee
that the key assumption of Theorem~\ref{theorem:main} -- weak localizability -- is fulfilled: Existence of a strong unit, of finite suprema and infima,
or of sufficiently many elements admitting a multiplicative inverse. These will be discussed in detail in the following subsections.
In order to guarantee that the equivalence relation $\canequiv$ on $\group{G}^+$ is equality, one usually needs the additional
assumption that the order is (weakly) archimedean (we mostly follow the conventions of the textbook \cite{goodearl:partiallyOrderedAbelianGroupsWithInterpolation}).

A partially ordered abelian group $\group{G}$ is called
\emph{archimedean} if the following holds: Whenever two elements
$a, b\in \group{G}$ fulfil $0\le \ell a + b$ for all $\ell\in \NN$, then $0 \le a$.
Similarly, $\group{G}$ is called \emph{weakly archimedean} if the following holds: Whenever two elements
$a, b\in \group{G}$ fulfil $-b \le \ell a \le b$ for all $\ell \in \NN$, then $a = 0$.
Every archimedean partially ordered abelian group is clearly weakly archimedean, and every archimedean directed 
partially ordered abelian group has an unperforated order by \cite[Prop.~1.24]{goodearl:partiallyOrderedAbelianGroupsWithInterpolation}.
Note that the archimedean properties are closely related to the operation $\argument^\ddagger$ from \eqref{eq:ddagger}:
A partially ordered abelian group $\group{G}$ is archimedean if and only if $(\group{G}^+)^\ddagger = \group{G}^+$,
and $\group{G}$ is weakly archimedean if and only if $(\group{G}^+)^\ddagger \cap (-(\group{G}^+)^\ddagger) = \{0\}$.

\begin{proposition} \label{proposition:orderedGrpCanequiv}
  Let $\group{G}$ be a partially ordered abelian group. If $\group{G}$ is weakly archimedean and its order relation $\le$ is unperforated
  (in particular, if $\group{G}$ is archimedean and directed), then the equivalence relation $\canequiv$ on $\group{G}^+$ is the equality relation.
\end{proposition}
\begin{proof}
  Assume $a,b\in \group{G}^+$ fulfil $a \canequiv b$. By definition and by the previous Proposition~\ref{proposition:orderedGrpCanleq}
  this means that there is $d \in \group{G}^+$ such that $ka \le kb + d$ and $kb \le ka + d$ for all $k\in \NN$.
  So $-d \le k(a-b) \le d$ for all $k\in \NN$ by translation-invariance, and therefore $a = b$ because $\group{G}$
  is weakly archimedean.
\end{proof}

\begin{example}
  Consider the abelian group $\ZZ^2$ with the translation-invariant and unperforated partial order with positive cone
  $\{(0,0)\} \cup \set{(m,n)}{m, n \in \NN}$. With this order, $\ZZ^2$ is weakly archimedean and directed,
  but it is not archimedean because $0 \le \ell (0,1) + (1,1)$ for all $\ell \in \NN$ but $0 \not\le (0,1)$.
  Conversely, the translation-invariant partial order on $\ZZ$ with positive cone $\NN_0 \setminus \{1\}$ is
  not unperforated, yet $\ZZ$ with this order is again weakly archimedean and directed.
\end{example}

\subsection{Order units}

Recall the definition of an order unit of a partially ordered abelian group $\group{G}$: An \emph{order-unit}
is an element $e \in \group{G}^+$ with the additional property that for all $a \in \group{G}$ there exists $k\in \NN$ such that $a \le ke$.
This implies in particular that $\group{G}$ is directed, because every element $a\in \group{G}$ can be expressed as
$a = ke - (ke-a)$ with $ke \in \group{G}^+$ and $ke-a \in \group{G}^+$ for sufficiently large $k\in \NN$.
For multiplications that admit an algebraic unit which at the same time is an order-unit, we obtain:

\begin{theorem} \label{theorem:orderunit}
  Let $\group{G}$ be a partially ordered abelian group and $\mu$ a biadditive binary operation on $\group{G}$ such that
  $\mu(a,b) \in \group{G}^+$ for all $a,b\in\group{G}^+$. Assume there exists a (necessarily unique) element $e \in \group{G}^+$
  that fulfils $\mu(e,a) = a = \mu(a,e)$ for all $a\in \group{G}$, and assume further that $e$ is also an order-unit of $\group{G}$.
  If $\group{G}$ is weakly archimedean and the order relation $\le$ unperforated (in particular, if $\group{G}$ is archimedean),
  then $\mu$ is both associative and commutative.
\end{theorem}
\begin{proof}
  Existence of an order-unit implies that $\group{G}$ is directed. So if $\group{G}$ is archimedean, then $\group{G}$
  is also weakly archimedean and its order relation $\le$ is unperforated.
  
  Now assume that $\group{G}$ is weakly archimedean and $\le$ unperforated. By Propositions~\ref{proposition:orderedGrpCanleq}
  and \ref{proposition:orderedGrpCanequiv}, the canonical quasi-order $\canleq$ and the equivalence relation $\canequiv$
  on $\group{G}^+$ are $\le$ and equality, respectively.
  The map $\mu$ restricts to a biadditive binary operation $\mu^+$ on $\group{G}^+$, and the unit $e$ and all its multiples
  $ke$ with $k\in \NN$ are $\mu^+$-localizable: Indeed, if $a,b\in \group{G}^+$ fulfil $\mu^+(ke,a)+a \le \mu^+(ke,b)+b$
  or $\mu^+(a,ke)+a \le \mu^+(b,ke)+b$, then this means $(k+1)a \le (k+1) b$, hence $a \le b$ by unperforatedness.
  As $e$ is an order-unit, this is enough for $\group{G}^+$ to be weakly $\mu^+$\=/localizable, so Corollary~\ref{corollary:equivIsEqual}
  applies and shows that $\mu^+$ is an associative and commutative operation on $\group{G}^+$. This implies that $\mu$ itself
  is both associative and commutative on the whole $\group{G}$ because $\group{G}$ is directed.
\end{proof}

This is (a strengthening of) the observation about automatic associativity and commutativity of archimedean ordered real algebras with order-unit made in 
\cite{vernikoff.krein.tovbin:surLesAnneauxSemiordonnes} and \cite[Sec.~3]{kadison:RepresentationTheoremForCommutativeTopologicalAlgebras},
which follows from a representation by continuous real-valued functions on a compact Hausdorff space.
Note that Theorem~\ref{theorem:orderunit} here does not require a vector space structure, and if the positive cone
is actually a convex cone of a real vector space, then Theorem~\ref{theorem:orderunit} only requires the assumption of a weakly
archimedean order because unperforatedness is clear in this case.

\subsection{Lattice orderings}

We proceed with an application to the lattice-ordered case:
A \emph{lattice-ordered abelian group} is a partially ordered abelian group $\group{G}$ with the extra property that the 
infimum $a\wedge b \coloneqq \inf \{a,b\}$
and supremum $a\vee b \coloneqq \sup \{a,b\}$ of any two elements $a,b\in \group{G}$ exist.
This way, $\wedge$ and $\vee$ are associative and commutative binary operations on $\group{G}$.
Translation-invariance of the order implies $(a\wedge b) + c = (a+c)\wedge(b+c)$ and 
$(a\vee b) + c = (a+c)\vee(b+c)$ for all $a,b,c\in \group{G}$, and one easily checks that
$-(a\wedge b) = (-a) \vee (-b)$ for all $a,b\in \group{G}$. In particular
$(a\vee b) - b = a - (a\wedge b)$.

Every lattice-ordered abelian group $\group G$ is directed. More precisely, given $a \in \group{G}$, then 
$a_+ \coloneqq a\vee 0 \in \group{G}^+$ and $a_- \coloneqq -(a\wedge 0) = (-a)\vee 0 \in \group{G}^+$, the
so-called \emph{positive} and \emph{negative parts} of $a$, fulfil $a + a_- = a_+$,
and therefore also $a_+ \wedge a_- = 0$ because $a_+ \wedge a_- = (a+a_-) \wedge a_- = (a\wedge 0) + a_- = 0$.

Moreover, if $\group{G}$ is a lattice-ordered abelian group, then its order relation $\le$ is unperforated,
see e.g.~\cite[Prop.~1.22]{goodearl:partiallyOrderedAbelianGroupsWithInterpolation} (essentially: if $2a \ge 0$, then 
$2(a\wedge 0) = (a + (a\wedge 0)) \wedge (0+(a\wedge 0)) = 2a \wedge a \wedge 0 = a \wedge 0$ shows that $a\wedge 0 = 0$, so $a \ge 0$; etc.).
It follows that $(ka) \wedge (kb) = k(a\wedge b)$ and $(ka) \vee (kb) = k(a\vee b)$ for all $k\in \NN$ and $a,b\in \group{G}$.
In particular the canonical quasi-order $\canleq$ on $\group{G}^+$ is the restriction of the given order $\le$ of $\group{G}$
by Proposition~\ref{proposition:orderedGrpCanleq}.

Consider a lattice-ordered abelian group $\group{G}$ and a biadditive binary operation $\mu$ on $\group{G}$
such that $\mu(a,b) \in \group{G}^+$ for all $a,b\in\group{G}^+$. Such a tuple $(\group{G},\mu)$ is called an
\emph{extended $f$\=/ring} if the following holds:
Whenever $a,b \in \group{G}$ fulfil $a\wedge b = 0$, then $\mu(c,a) \wedge b = \mu(a,c)\wedge b = 0$ for all
$c\in\group{G}^+$.
An archimedean extended $f$\=/ring is an extended $f$\=/ring $(\group{G},\mu)$ with archimedean $\group{G}$.
We use the notion ``extended'' $f$\=/ring to stress the fact that neither associativity nor commutativity are required by definition.

\begin{lemma} \label{lemma:frng}
  Let $\group{G}$ be a lattice-ordered abelian group and assume that the three elements $a,b,c \in \group{G}^+$
  fulfil $a \le b+c$, then also $a \le (b \wedge a) + (c\wedge a)$.
\end{lemma}
\begin{proof}
  From $a - (b\wedge a) = (a-b) \vee 0 \le c$ and $a - (b\wedge a) \le a$ it follows that $a - (b\wedge a) \le c \wedge a$.
\end{proof}

\begin{proposition} \label{proposition:fring}
  Let $(\group{G},\mu)$ be an extended $f$\=/ring and $\mu^+ \colon \group G^+ \times \group G^+ \to \group G^+$ the restriction of $\mu$ to $\group G^+$.
  Then $\group{G}^+$ is strongly $\mu^+$\=/localizable.
\end{proposition}
\begin{proof}
  By Proposition~\ref{proposition:orderedGrpCanleq} the canonical quasi-order $\canleq$ on $\group{G}^+$ is the restriction of the order $\le$ of $\group{G}$.
  
  Let $s\in \group G^+$ and $a,b\in \group G^+$ be given such that $\mu^+(s,a) + a \le \mu^+(s,b)+b$.
  Set $c \coloneqq b-a \in \group{G}$ and let $c_+,c_-\in \group{G}^+$
  be the positive and negative parts of $c$.
  Then $c_+ \wedge c_- = 0$ and therefore $\mu(s,c_+)\wedge c_- = 0$ because $(\group{G},\mu)$ is an extended $f$\=/ring.
  Moreover, from $c_+ - c_- = b-a$ it follows that $\mu(s,c_-)+c_- \le \mu(s,c_+) + c_+$.
  Therefore the estimate $0 \le c_- \le \mu(s,c_-) + c_- \le \mu(s,c_+) + c_+$ holds.
  The previous Lemma~\ref{lemma:frng} shows that even
  $0 \le c_- \le \bigl(\mu(s,c_+) \wedge c_-\bigr) + (c_+\wedge c_-)$. As the right-hand side of this estimate
  reduces to $0$, this results in $c_- = 0$, so $c = c_+ \ge 0$ and therefore $a \le b$.
  This shows that $s$ is left $\mu^+$\=/localizable. The same argument with $\mu^\op$ in place of $\mu$ then shows that
  $s$ is $\mu^+$\=/localizable.
\end{proof}

We thus obtain the automatic associativity and commutativity of archimedean extended $f$\=/rings as in 
\cite{amemiya:falgebrasApparently, bernau:onSemiNormalLatticeRings, birkhoff.pierce:latticeOrderedRings}:

\begin{theorem} \label{theorem:fring}
  Let $(\group{G},\mu)$ be an archimedean extended $f$\=/ring, then $\mu$ is both associative and commutative.
\end{theorem}
\begin{proof}
  By Proposition~\ref{proposition:orderedGrpCanequiv}, the equivalence relation $\canequiv$ on $\group{G}^+$ is the equality relation.
  Due to the previous Proposition~\ref{proposition:fring} we can apply Corollary~\ref{corollary:equivIsEqual} to the restriction
  $\mu^+ \colon \group G^+ \times \group G^+ \to \group G^+$ of $\mu$ to $\group{G}^+$, so $\mu^+$ is both associative and commutative.
  As $\group{G}$ is directed, this implies that $\mu$ itself is both associative and commutative on the whole $\group{G}$.
\end{proof}

Note that one does not obtain a more general result by reformulating Theorem~\ref{theorem:fring} for weakly
archimedean extended $f$\=/rings, because every weakly archimedean lattice-ordered abelian group $\group{G}$ is archimedean:
Indeed, if $0 \le \ell a + b$ holds for $a,b\in \group{G}$ and $\ell \in \NN$,
then $\ell a_- \le \ell a_+ + b$ and therefore $\ell a_- \le (\ell a_+) \wedge (\ell a_-) + b \wedge (\ell a_-)$ by Lemma~\ref{lemma:frng}.
But $(\ell a_+) \wedge (\ell a_-) = \ell(a_+\wedge a_-) = 0$ and $b \wedge (\ell a_-) \le b \le b_+$, so $-b_+ \le 0 \le \ell a_- \le b_+$ holds.
If $\group G$ is weakly archimedean and $0 \le \ell a + b$ for two elements $a,b\in \group{G}$ and all $\ell \in \NN$,
then $-b_+ \le \ell a_- \le b_+$ for all $\ell \in \NN$ implies $a_- = 0$, thus $0 \le a_+ = a$.

There is also a seemingly related result concerning the automatic commutativity of extended almost $f$\=/rings:
An \emph{extended almost $f$\=/ring} is a tuple $(\group{G},\mu)$ of a lattice-ordered abelian group $\group{G}$ and a biadditive
binary relation $\mu$ on $\group G$ fulfilling $\mu(a,b) \in \group{G}^+$ for all $a,b\in \group G^+$ and such that
the following condition holds: Whenever $a,b \in \group{G}$ fulfil $a\wedge b = 0$, then $\mu(a,b) = 0$.
Every associative archimedean extended almost $f$\=/ring $(\group{G},\mu)$ is automatically commutative, see 
\cite{basly.triki:ffalgebresArchimediennesReticulees, bernau.huijsmans:almostFAlgebrasAndDAlgebras, scheffold:ffBanachverbandsalgebren}.
At least if $\group G$ also carries the structure of a real vector space and if the scalar multiplication with positive
real numbers preserves the order, then it is not necessary to assume associativity of $\mu$ in order to prove commutativity
of $\mu$, see \cite{buskes.vanRooij:AlmostFAlgebrasCommutativityandCSInequality}. However, there do exist archimedean extended almost $f$\=/rings
that are not associative, an example has been given in \cite{bernau.huijsman:onSomeClassesOfLatticeOrderedAlgebras}: 
Let $\group{G}$ be the the archimedean lattice-ordered abelian
group of all continuous real-valued functions on the compact interval $[0,1]$ with the pointwise addition and pointwise comparison,
and let $\mu \colon \group{G} \times \group{G} \to \group{G}$ be defined as $\mu(a,b)(x) \coloneqq a(0) b(0) + a(1) b(1)$ for all $x\in {[0,1]}$.
This example demonstrates that the automatic commutativity of extended almost $f$\=/rings is not a direct consequence of Theorem~\ref{theorem:main}
(which would also imply associativity), but is an independent result of different nature.

\subsection{Division rings}

We finally apply Theorem~\ref{theorem:main} to division rings. An \emph{ordered division ring} is a (by definition associative) division ring $\FF$ endowed with
a total order $\le$ such that the set $\FF^+ \coloneqq \set{a\in \FF}{0 \le a}$ is closed under addition and multiplication
and fulfils $a^2 \in \FF^+$ for all $a\in \FF$. For any division ring $\FF$ we define
\begin{align}
  \FF^{**} \coloneqq \genMonoid[\big]{\set[\big]{(a_1)^2 \dots (a_n)^2}{n\in \NN;\,a_1,\dots,a_n \in \FF}},
\shortintertext{where}
  \genMonoid{S} \coloneqq \set[\Big]{\sum\nolimits_{j=1}^k s_j }{k\in \NN_0; \,s_1, \dots, s_k \in S }
\end{align}
is the submonoid of the additive group of $\FF$ generated by a subset $S$ of $\FF$.
Clearly $\FF^{**}$ is the smallest subset of $\FF$ that is closed under addition and multiplication and that contains all
squares of elements of $\FF$. Generalizing from the commutative case, a division ring $\FF$ is called \emph{formally real} if
$-1 \notin \FF^{**}$. By \cite{szele:orderedSkewFields}, a division ring $\FF$ admits an order relation $\le$ that turns
$\FF$ into an ordered division ring if and only if $\FF$ is formally real. Every formally real division ring $\FF$ clearly has characteristic $0$.

Every weakly archimedean ordered division ring $\FF$ is archimedean by totality of the order.
It is also easy to see that a division ring $\FF$ is archimedean if and only if for every $a \in \FF$ there exists $k\in \NN$ such that $a-k \notin \FF^+$.
It has long been known that every archimedean ordered division ring $\FF$ is commutative and therefore is a subfield of 
the real numbers $\RR$, see \cite[Sec.~32]{hilbert:grundlagenDerGeometrie};
but there do exist non-commutative formally real division rings $\FF$, see \cite[Sec.~33]{hilbert:grundlagenDerGeometrie},
in which case all compatible orderings on $\FF$ are necessarily non-archimedean. 

The following result about automatic commutativity of certain formally real division rings generalizes this classical
theorem:

\begin{theorem} \label{theorem:skew}
  Let $\FF$ be a division ring and assume that for all $a \in \FF$ there is $k\in \NN$ such that 
  $a-k \notin \FF^{**}$. Then the multiplication of $\FF$ is commutative and $\FF$ is a formally real field.
\end{theorem}
\begin{proof}
  If $-1 \in \FF^{**}$, then $-k \in \FF^{**}$ for all $k\in \NN$, contradicting the assumption for $a = 0$.
  So $-1 \notin \FF^{**}$, i.e.~$\FF$ is formally real, and in particular $\FF$ has characteristic $0$.
  Recall that
  \begin{equation*}
    \bigl(\FF^{**}\bigr)^\ddagger = \set[\big]{b \in \FF}{ \textup{there exists }a\in \FF\textup{ such that }a+k b \in \FF^{**}\textup{ for all }k \in \NN}
    .
  \end{equation*}
  Clearly $(\FF^{**})^\ddagger$ is closed under addition, and by assumption $-1 \notin (\FF^{**})^\ddagger$.
  Moreover, given $b \in (\FF^{**})^\ddagger$
  and $c \in \FF^{**}$, then $bc \in (\FF^{**})^\ddagger$: Indeed, there exists $a \in \FF$ such that $a+k b \in \FF^{**}$ for all $k \in \NN$,
  and then also $ac+k bc = (a+k b)c \in \FF^{**}$ for all $k \in \NN$.
  
  Consider $I \coloneqq (\FF^{**})^\ddagger \cap \bigl({-(\FF^{**})^\ddagger}\bigr)$, then $I$ is a right ideal
  of $\FF$: It is easy to check that $I$ is an additive subgroup because $(\FF^{**})^\ddagger$ is closed under addition;
  given $b \in I$, $c \in \FF$, then $b,-b \in (\FF^{**})^\ddagger$
  and $c = ((c+1)/2)^2 - ((c-1)/2)^2$, therefore $bc = b ((c+1)/2)^2 + (-b)((c-1)/2)^2 \in (\FF^{**})^\ddagger$ and 
  $-bc = (-b) ((c+1)/2)^2 + b((c-1)/2)^2 \in (\FF^{**})^\ddagger$, so $bc \in I$. As $\FF$ is a division ring, this means that
  either $I = \{0\}$ or $I = \FF$. But the latter would imply $-1 \in I \subseteq (\FF^{**})^\ddagger$, a contradiction,
  so $I = \{0\}$.
  
  The multiplication of $\FF$ clearly restricts
  to a biadditive binary operation $\mu \colon \FF^{**} \times \FF^{**} \to \FF^{**}$, $(a,b)\mapsto \mu(a,b) \coloneqq ab$.
  We want to apply Theorem~\ref{theorem:main} and its corollaries to $\FF^{**}$ and $\mu$:
  
  First consider $a,b\in \FF^{**}$ such that $a \canleq b$. This means there exist $c,t \in \FF^{**}$ and $k\in \NN$
  such that $ka+c+t = kb + t$, so $b-a = c/k = (kc)(k^{-1})^2 \in \FF^{**}$.
  Therefore the canonical quasi-order $\canleq$ on $\FF^{**}$ is simply given, for $a,b\in \FF^{**}$, by $a \canleq b$ if and only if $b-a \in \FF^{**}$.
  Similarly, if $a,b \in \FF^{**}$ fulfil $a \canequiv b$, then this means there exists $d \in \FF^{**}$
  such that $ka \canleq kb + d$ and $kb \canleq ka + d$ hold for all $k\in \NN$,
  or equivalently $d+k(b-a) \in \FF^{**}$ and $d+k(a-b) \in \FF^{**}$ for all $k\in \NN$, so $b-a \in I$, i.e.~$b=a$.
  The canonical equivalence relation $\canequiv$ therefore is the equality relation.
  
  Next we show that $\FF^{**}$ is strongly $\mu$\=/localizable, so consider $s \in \FF^{**}$:
  From $- 1 \notin \FF^{**}$ it follows that $1 + s \neq 0$, so the multiplicative inverse of $1 +s$ exists and
  $(1+s)^{-1} = ((1+s)^{-1})^2 (1+s) \in \FF^{**}$. If $a,b \in \FF^{**}$
  fulfil $\mu(s,a)+ a \canleq \mu(s,b) + b$ or $\mu(a,s)+ a \canleq \mu(b,s) + b$,
  then this means that $(1+s)(b-a) \in \FF^{**}$ or $(b-a)(1+s) \in \FF^{**}$, respectively.
  In both cases it follows that $b-a \in \FF^{**}$ by multiplication with $(1+s)^{-1} \in \FF^{**}$,
  so $a \canleq b$.
  
  We can finally apply Corollary~\ref{corollary:equivIsEqual} to the operation $\mu$ on $\FF^{**}$, which shows that
  $\mu$ is commutative (associativity of $\mu$ holds by assumption). As every element $a \in \FF$ is the difference
  of two elements of $\FF^{**}$, e.g.~$a = ((a+1)/2)^2 - ((a-1)/2)^2$,
  it also follows that the multiplication of $\FF$ is commutative. So $\FF$ is a field, and $\FF$ is formally real
  as discussed above.
\end{proof}
Theorem~\ref{theorem:skew} essentially says that every weakly archimedean \emph{partially} ordered skew field is commutative:

\begin{corollary}
  Let $\FF$ be a division ring and assume there exists a translation-invariant partial order $\le$ on $\FF$ such that the positive cone
  $\FF^+ \coloneqq \set{a\in \FF}{0 \le a}$ is closed under multiplication and $a^2 \in \FF^+$ for all $a \in \FF$. If the
  partially ordered abelian group $(\FF,+,\le)$ is weakly archimedean, then the multiplication of $\FF$ is commutative and $\FF$ is a formally real field.
\end{corollary}
\begin{proof}
  Clearly $\FF^{**} \subseteq \FF^+$. Assume there exists $a \in \FF$ such that $a-k \in \FF^{**}$ for all $k\in \NN$,
  then $k \le a$ for all $k\in \NN$. It follows that
  $-a \le 0 \le k \le a$ for all $k\in \NN$, so the partially ordered abelian group $(\FF,+,\le)$ is not weakly archimedean.
  Passing to the contrapositive, if $(\FF,+,\le)$ is weakly archimedean, then for all $a\in \FF$ there exists $k\in \NN$
  such that $a-k \notin \FF^{**}$. In this case the previous Theorem~\ref{theorem:skew} applies and shows that 
  the multiplication of $\FF$ is commutative and that $\FF$ is a formally real field.
\end{proof}
The next example shows that Theorem~\ref{theorem:skew} is strictly more general than the classical result that all archimedean
ordered division rings are commutative, and that it cannot be proven by simply noting that the assumption of Theorem~\ref{theorem:skew}
mean that $\FF$ can be (totally) ordered and by then applying the classical result to one of the compatible orders:

\begin{example} \label{example:rat}
  Consider the field of rational functions $\RR(x)$, which is of course formally real. Every $a \in \RR(x)^{**}$ 
  clearly is pointwise positive almost everywhere on $\RR$, from which it follows that for every $a \in \RR(x)$ there exists
  $k \in \NN$ such that $a- k \notin \RR(x)^{**}$, i.e.~Theorem~\ref{theorem:skew} applies. Yet whatever total order one 
  chooses in order to turn $\RR(x)$ into an ordered field, the result will always be non-archimedean because $\RR(x)$ cannot be 
  realised as a subfield of $\RR$.
\end{example}

\begin{remark}
  Any division ring $\FF$ falls into exactly one of the following three categories:
  \begin{enumerate}
    \item $-1 \in \FF^{**}$. Such division rings cannot be ordered. E.g.~the quaternions fall into this category.
    \item $-1 \notin \FF^{**}$ but $-1 \in (\FF^{**})^\ddagger$. Such division rings can be ordered by \cite{szele:orderedSkewFields},
      and they may be non-commutative. A non-commutative example in this category is given in e.g.~\cite[Sec.~33]{hilbert:grundlagenDerGeometrie}.
    \item $-1 \notin (\FF^{**})^\ddagger$. This is just the assumption of Theorem~\ref{theorem:skew}
      (note that $-1 \in (\FF^{**})^\ddagger$ means that there exists $a\in \FF$ such that $a-k \in \FF^{**}$ for all $k\in \NN$).
      Such division rings therefore are always commutative.
      E.g.~the rational functions fall into this category, see the previous Example~\ref{example:rat}.
  \end{enumerate}
  So non-commutative orderable division rings walk a narrow line as they need to fulfil $-1 \in (\FF^{**})^\ddagger \setminus \FF^{**}\!$.
\end{remark}

\section{On weak and strong localizability} \label{sec:loc}

As a final remark, we show that in well-behaved situations, weak localizability implies strong localizability.
We continue using the notions for partially ordered abelian groups:

\begin{theorem}
  Let $\group{G}$ be an archimedean directed partially ordered abelian group and let $\mu$ be
  a biadditive binary operation on $\group{G}$ such that $\mu(a,b) \in \group{G}^+$ for all $a,b\in\group{G}^+$.
  Let $\mu^+ \colon \group{G}^+ \times \group{G}^+ \to \group{G}^+$ be the restriction of $\mu$ to $\group{G}^+$.
  If $\group{G}^+$ is weakly $\mu^+$\=/localizable, then $\group{G}^+$ is strongly $\mu^+$\=/localizable.
\end{theorem}
\begin{proof}
  As $\group{G}$ is archimedean and directed, its order $\le$ is unperforated, so by Propositions~\ref{proposition:orderedGrpCanleq}
  and \ref{proposition:orderedGrpCanequiv}, the canonical order $\canleq$ and the equivalence relation $\canequiv$ on $\group{G}^+$
  are $\le$ and equality, respectively.
  In the following we will apply Proposition~\ref{proposition:mult}; as a preparation we show that every
  left $\mu^+$\=/localizable element $s \in \group{G}^+$ fulfils the localizability condition \eqref{eq:mult:localizability}:
  
  Assume $a \in \group{G}$ fulfils $0 \le \mu(s,a) + a$. As $\group{G}$ is directed there are $b,c \in \group{G}^+$
  such that $a = b-c$ and then $\mu^+(s,c) + c \le \mu^+(s,b) + b$. From left $\mu^+$\=/localizability of $s$
  it follows that $c \le b$, i.e.~$a \in \group{G}^+$.
  
  Now assume that $\group{G}^+$ is weakly $\mu^+$\=/localizable and let an arbitrary element
  $r \in \group{G}^+$ be given. It is sufficient to show that $r$ is left $\mu^+$\=/localizable,
  because by Corollary~\ref{corollary:equivIsEqual}, $\mu^+$ is commutative anyway.
  So consider $a,b\in\group{G}^+$ such that $\mu^+(r,a) + a \le \mu^+(r,b)+b$.
  The proof is complete once we have shown that $a \le b$.
  
  There exists a $\mu^+$\=/localizable element $s \in \group{G}^+$ such that $a+b+r \le s$.
  Set $F \coloneqq \{ a,b,r,s\}$, then $0 \le f \le s$ for all $f\in F$.
  Set $\group{H} \coloneqq \genGrp{F \cup \set{\mu(f,f')}{f,f' \in F}}$ and let $\varphi$ be any extremal positive additive functional on $\group{H}$.
  We show that $\varphi(a) \le \varphi(b)$:
  Clearly $\varphi(s) \ge 0$ and $\varphi(\mu(s,s)) \ge 0$ by positivity of $\varphi$.
  If $\varphi(s) = 0$, then $\varphi(f) = 0$ for all $f\in F$ because $0 \le \varphi(f) \le \varphi(s) = 0$.
  In particular $\varphi(a) = 0 \le \varphi(b)$.
  If $\varphi(\mu(s,s)) = 0$, then $\varphi(\mu(r,f)) = 0$ for all $f\in F$ because $0 \le \varphi(\mu(r,f)) \le \varphi(\mu(r,s)) \le \varphi(\mu(s,s)) = 0$.
  In particular $\varphi(\mu(r,a)) = 0 = \varphi(\mu(r,b))$ and therefore again
  $\varphi(a) = \varphi(\mu(r,a)+a) \le \varphi(\mu(r,b)+b) = \varphi(b)$.
  Otherwise, i.e.~if $\varphi(s) > 0$ and $\varphi(\mu(s,s)) > 0$, Proposition~\ref{proposition:mult} applies to $F$
  and shows that the identity $\varphi(s)\,\varphi(\mu(f,f')) = \varphi(\mu(f,s))\, \varphi(f')$
  holds for all $f,f' \in F$. It then follows
  that
  \begin{equation*}
    \bigl(\varphi(\mu(r,s)) + \varphi(s)\bigr) \varphi(a)
    =
    \varphi(s)\,\varphi\bigl( \mu(r,a) + a \bigr)
    \le
    \varphi(s)\,\varphi\bigl( \mu(r,b) + b \bigr)
    =
    \bigl(\varphi(\mu(r,s)) + \varphi(s)\bigr) \varphi(b)
  \end{equation*}
  and therefore $\varphi(a) \le \varphi(b)$ because $\varphi(\mu(r,s)) + \varphi(s) \ge \varphi(s) > 0$.
  
  Set $e \coloneqq \mu(s,s) + s \in \group{H} \cap \group{G}^+$, then $f \le s \le e$ holds for all $f\in F$,
  and similarly also $\mu(f,f') \le \mu(f,s) \le \mu(s,s) \le e$ for all $f,f' \in F$. Every non-zero extremal positive
  additive functional $\varphi$ on $\group{H}$ therefore fulfils $0 < \varphi(e)$ and consequently
  $0 < \varphi(e) \le \varphi\bigl(\ell(b-a) + e\bigr)$ for all $\ell\in \NN$. Lemma~\ref{lemma:positivstellensatz},
  applied to the finite subset $F \cup \set{\mu(f,f')}{f,f'\in F}$ of $\group{G}^+$\!, now 
  shows that for every $\ell \in \NN$ there exists $k\in \NN$ such that $0 \le k\bigl(\ell(b-a) + e\bigr)$,
  and even $0 \le \ell(b-a) + e$ for all $\ell \in\NN$ because the order $\le$ is unperforated.
  As $\group{G}$ is archimedean by assumption it follows that $b-a \in \group{G}^+\!$, i.e.~$a\le b$.
\end{proof}

\section*{Acknowledgements}
I would like to thank the anonymous referee for many very detailed remarks and suggestions that helped improve readability
and overall quality of the manuscript.

\end{onehalfspace}

\begin{thebibliography}{10}

\bibitem {amemiya:falgebrasApparently}
{Amemiya, I.: }\newblock \emph{A general spectral theory in semi-ordered
  linear spaces}.
\newblock J. Fac. Sc. Hokkaido Un. Ser. I  \textbf{12} (1953), 111--156.

\bibitem {basly.triki:ffalgebresArchimediennesReticulees}
{Basly, M., Triki, A.: }\newblock \emph{FF-alg{\'e}bres
  Archim{\'e}diennes r{\'e}ticul{\'e}es}.
\newblock University of Tunis (preprint)   (1988).

\bibitem {bernau:onSemiNormalLatticeRings}
{Bernau, S.~J.: }\newblock \emph{On semi-normal lattice rings}.
\newblock In: \emph{Mathematical Proceedings of the Cambridge Philosophical
  Society}, vol.~61,   613--616. Cambridge University Press, 1965.
\newblock \href {https://doi.org/10.1017/S0305004100038949}
  {\path{doi:10.1017/S0305004100038949}}.

\bibitem {bernau.huijsman:onSomeClassesOfLatticeOrderedAlgebras}
{Bernau, S.~J., Huijsmans, C.~B.: }\newblock \emph{On Some Classes of
  Lattice-Ordered Algebras}.
\newblock In: \emph{Ordered Algebraic Structures: Proceedings of the Caribbean
  Mathematics Foundation Conference on Ordered Algebraic Structures,
  Cura{\c{c}}ao, August 1988},   175--179, Dordrecht, 1989. Springer
  Netherlands.
\newblock \href {https://doi.org/10.1007/978-94-009-2472-7_13}
  {\path{doi:10.1007/978-94-009-2472-7_13}}.

\bibitem {bernau.huijsmans:almostFAlgebrasAndDAlgebras}
{Bernau, S.~J., Huijsmans, C.~B.: }\newblock \emph{Almost f-algebras and
  d-algebras}.
\newblock In: \emph{Mathematical Proceedings of the Cambridge Philosophical
  Society}, vol. 107,   287--308. Cambridge University Press, 1990.
\newblock \href {https://doi.org/10.1017/S0305004100068560}
  {\path{doi:10.1017/S0305004100068560}}.

\bibitem {birkhoff.pierce:latticeOrderedRings}
{Birkhoff, G., Pierce, R.: }\newblock \emph{Lattice-ordered rings}.
\newblock An. Acad. Brasil. Ci  \textbf{28} (1956).

\bibitem {bucy.maltese:RepresentationTheoremForPositiveFunctionals}
{Bucy, R.~S., Maltese, G.: }\newblock \emph{A Representation Theorem for
  Positive Functionals on Involution Algebras}.
\newblock Mathematische Annalen  \textbf{162} (1965/66), 364--368.
\newblock \href {https://doi.org/10.1007/BF01369109}
  {\path{doi:10.1007/BF01369109}}.

\bibitem {buskes.vanRooij:AlmostFAlgebrasCommutativityandCSInequality}
{Buskes, G., van Rooij, A.: }\newblock \emph{Almost f-algebras:
  Commutativity and the Cauchy-Schwarz Inequality}.
\newblock Positivity  \textbf{4}.3 (2000), 227--231.
\newblock \href {https://doi.org/10.1023/A:1009826510957}
  {\path{doi:10.1023/A:1009826510957}}.

\bibitem {cimpric.marshall.netzer:closuresOfQuadraticModules}
{Cimpri{\v{c}}, J., Marshall, M., Netzer, T.: }\newblock \emph{Closures
  of quadratic modules}.
\newblock Israel Journal of Mathematics  \textbf{183} (2011), 445--474.
\newblock \href {https://doi.org/10.1007/s11856-011-0056-y}
  {\path{doi:10.1007/s11856-011-0056-y}}.

\bibitem {goodearl:partiallyOrderedAbelianGroupsWithInterpolation}
{Goodearl, K.~R.: }\newblock \emph{Partially Ordered Abelian Groups With
  Interpolation}.
\newblock American Mathematical Society, 1986.

\bibitem {hilbert:grundlagenDerGeometrie}
{Hilbert, D.: }\newblock \emph{Grundlagen der Geometrie}.
\newblock Teubner, 1913.

\bibitem {kadison:RepresentationTheoremForCommutativeTopologicalAlgebras}
{Kadison, R.~V.: }\newblock \emph{A Representation Theory for
  Commutative Topological Algebra}.
\newblock Memoirs of the American Mathematical Society, 1951.

\bibitem {kuhlmann.marshall:positivitySOSandMultidimensionalMomentProblem}
{Kuhlmann, S., Marshall, M.: }\newblock \emph{Positivity, sums of
  squares and the multi-dimensional moment problem}.
\newblock Transactions of the American Mathematical Society  \textbf{354}.11
  (2002), 4285--4301.
\newblock \href {https://doi.org/10.1515/advg.2005.5.4.583}
  {\path{doi:10.1515/advg.2005.5.4.583}}.

\bibitem {scheffold:ffBanachverbandsalgebren}
{Scheffold, E.: }\newblock \emph{FF-Banachverbandsalgebren}.
\newblock Mathematische Zeitschrift  \textbf{177}.2 (1981), 193--205.
\newblock \href {https://doi.org/10.1007/BF01214199}
  {\path{doi:10.1007/BF01214199}}.

\bibitem {schmuedgen.schoetz:positivstellensaetzForSemirings}
{Schm{\"u}dgen, K., Sch{\"o}tz, M.: }\newblock
  \emph{Positivstellens\"atze for Semirings}.
\newblock Mathematische Annalen  \textbf{389} (2024), 947--985.
\newblock \href {https://doi.org/10.1007/s00208-023-02656-0}
  {\path{doi:10.1007/s00208-023-02656-0}}.

\bibitem {schoetz:gelfandNaimarkTheorems}
{Sch\"otz, M.: }\newblock \emph{{G}elfand--{N}aimark Theorems for
  Ordered *-Algebras}.
\newblock Canadian Journal of Mathematics  \textbf{75}.4 (2023), 1272--–1292.
\newblock \href {https://doi.org/10.4153/S0008414X22000359}
  {\path{doi:10.4153/S0008414X22000359}}.

\bibitem {szele:orderedSkewFields}
{Szele, T.: }\newblock \emph{On Ordered Skew Fields}.
\newblock Proceedings of the American Mathematical Society  \textbf{3}.3
  (1952), 410--413.
\newblock \href {https://doi.org/10.2307/2031894} {\path{doi:10.2307/2031894}}.

\bibitem {vernikoff.krein.tovbin:surLesAnneauxSemiordonnes}
{Vernikoff, I., Krein, S., Tovbin, A.: }\newblock \emph{Sur les anneaux
  semi-ordonn{\'e}s}.
\newblock In: \emph{CR (Doklady) Acad. Sci. URSS (NS)}, vol.~30,   785--787,
  1941.

\end{thebibliography}
\end{document}